\documentclass[12pt]{amsart}

\usepackage{amsmath,amssymb,amsfonts}
\usepackage{mathrsfs}
\usepackage{mathtools}
\usepackage{enumitem}
\usepackage{xcolor}
\usepackage{tikz-cd}
\usepackage{hyperref}
\usepackage[alphabetic,msc-links]{amsrefs}
\hypersetup{
  colorlinks=false,
  pdfpagemode=UseNone,
  pdfborder={0 0 1},
  pdfborderstyle={/S/D/D[3 2]},
  citebordercolor={0 0.55 0},
  linkbordercolor={0 0.55 0},
  urlbordercolor={0 0.55 0},
  pdftitle={Cheeger Constant Rigidity for Cocompact Negatively Curved Manifolds},
  pdfauthor={Kuntao Jin, Xiaodong Wang, and Bo Zhu},
  pdfsubject={Cheeger rigidity under an upper sectional-curvature bound},
  pdfkeywords={Cheeger constant, Busemann function, horospherical suspension, strong stable foliation, rigidity}
}

\usepackage{geometry}
\newtheorem{theorem}{Theorem}[section]

\newtheorem{proposition}[theorem]{Proposition}
\newtheorem{lemma}[theorem]{Lemma}
\newtheorem{corollary}[theorem]{Corollary}

\newtheorem*{acknowledgements}{Acknowledgements}

\theoremstyle{definition}
\newtheorem{definition}[theorem]{Definition}
\newtheorem{example}[theorem]{Example}

\numberwithin{equation}{section}

\DeclareMathOperator{\area}{area}
\DeclareMathOperator{\vol}{vol}
\DeclareMathOperator{\supp}{supp}
\DeclareMathOperator{\diver}{div}

\newcommand{\R}{\mathbb R}

\newcommand{\Wss}{\mathcal W^{ss}}
\newcommand{\Wws}{\mathcal W^{ws}}
\newcommand{\tWss}{\widetilde{\mathcal W}^{ss}}
\newcommand{\tWws}{\widetilde{\mathcal W}^{ws}}
\newcommand{\divss}{\operatorname{div}^{ss}}

\newcommand{\Hh}{\mathbb{H}}
\newcommand{\secg}{\operatorname{sec}}
\newcommand{\hiso}{h_{\mathrm{iso}}}

\begin{document}

\title[Cheeger constant rigidity]{Cheeger Constant Rigidity for Cocompact\\
Negatively Curved Manifolds}
\date{\today}

\author{Kuntao Jin}
\address[Kuntao Jin]{Department of Mathematical Sciences, Tsinghua University}
\email{jkt25@mails.tsinghua.edu.cn}

\author{Xiaodong Wang}
\address[Xiaodong Wang]{Department of Mathematics, Michigan State University,
East Lansing, MI 48824}
\email{xwang@math.msu.edu}

\author{Bo Zhu}
\address[Bo Zhu]{Yau Mathematical Sciences Center,
Tsinghua University}
\email{zhub@tsinghua.edu.cn}
\thanks{Bo Zhu is supported by NSFC~12501066.}


\begin{abstract}
Let $(M^m,g)$, $m\geq2$, be a closed connected Riemannian manifold with
$\secg_g\leq-1$, and let $(X,g_X)$ be its universal cover.  Yau proved
that $\hiso(X)\geq m-1$.  We prove that equality is rigid: if
$\hiso(X)=m-1$, then $(X,g_X)$ is isometric to $\Hh^m(-1)$.
\end{abstract}

\maketitle

\setcounter{tocdepth}{2}
\section{Introduction}
\label{sec:introduction}

Let $(M^m,g)$, $m\geq2$, be a closed connected Riemannian manifold with
$\secg_g\leq-1$.  Denote its universal cover and deck group by
\[
  (X,g_X)=(\widetilde M,\widetilde g),
  \qquad
  \Gamma=\pi_1(M).
\]
Let $d_X$ be the Riemannian distance on $X$.  Throughout,
$\Delta=\diver\nabla$; the symbol $\cong$ denotes a Riemannian isometry.
The Cheeger isoperimetric constant is
\[
  \hiso(X)
  :=
  \inf_{\Omega\Subset X}
  \frac{\area(\partial\Omega)}{\vol(\Omega)},
\]
where the infimum is taken over relatively compact domains with smooth
boundary
(see \cite{Cheeger1970}*{definition of the isoperimetric constant}).

For every $\xi\in\partial_\infty X$, the Busemann function $B_\xi$ is
of class $C^2$, and Busemann comparison gives
\[
  |\nabla B_\xi|=1,
  \qquad
  \Delta B_\xi\geq m-1.
\]
If $\Omega\Subset X$ has smooth boundary and outward unit normal $\nu$,
the divergence theorem gives
\[
  (m-1)\vol(\Omega)
  \leq\int_\Omega\Delta B_\xi\,dV
  =\int_{\partial\Omega}
    \langle\nabla B_\xi,\nu\rangle\,dA
  \leq\area(\partial\Omega).
\]
Taking the infimum over $\Omega$ gives Yau's sharp estimate
\begin{equation}\label{eq:yau-lower}
  \hiso(X)\geq m-1
  =\hiso\bigl(\Hh^m(-1)\bigr).
\end{equation}
More generally, Yau proved $\hiso(Y)\geq(m-1)a$ for every complete simply
connected $m$-manifold $(Y,g_Y)$ with $\secg_{g_Y}\leq-a^2<0$; see
\cite{Yau1975Isoperimetric}*{Proposition~3 and Corollary~1}.
Theorem~\ref{thm:main} is the equality case of \eqref{eq:yau-lower}.

\begin{theorem}\label{thm:main}
Let $(M^m,g)$, $m\geq2$, be a closed connected Riemannian manifold with
$\secg_g\leq-1$, and let $(\widetilde M,\widetilde g)$ be its universal
Riemannian cover.  If $\hiso(\widetilde M)=m-1$, then
\[
  (\widetilde M,\widetilde g)\cong\Hh^m(-1).
\]
\end{theorem}

Since equality holds on $\Hh^m(-1)$, Theorem~\ref{thm:main} and
\eqref{eq:yau-lower} give the equivalent formulation
\[
  \hiso(X)=m-1
  \quad\Longleftrightarrow\quad
  (X,g_X)\cong\Hh^m(-1).
\]

\subsection*{Relations among \texorpdfstring{$h_{\mathrm{iso}}$,
$\lambda_1$, and $h_{\mathrm{vol}}$}{the Cheeger constant, bottom spectrum,
and volume entropy}}

Set
\[
  \lambda_1(X):=\inf_{0\neq u\in C_c^\infty(X)}
  \frac{\int_X|\nabla u|^2\,dV}{\int_Xu^2\,dV},
  \qquad
  h_{\mathrm{vol}}(X):=\lim_{R\to\infty}\frac{1}{R}
  \log\vol B_X(o,R).
\]
The second limit exists and is independent of $o$ because $X$ is the universal
cover of a closed manifold.  Cheeger's inequality and Brooks's volume-growth
estimate give the following comparison
(see \cite{Cheeger1970}, \cite{Brooks1981}, and
\cite{JinZhu2026McKean}*{Introduction}):
\begin{equation}\label{eq:intro-cheeger-spectrum-entropy}
  h_{\mathrm{iso}}(X)^2
  \leq4\lambda_1(X)
  \leq h_{\mathrm{vol}}(X)^2.
\end{equation}

Since $h_{\mathrm{iso}}(X)\geq m-1$ by Yau's estimate,
\eqref{eq:intro-cheeger-spectrum-entropy} gives
\[
  h_{\mathrm{vol}}(X)=m-1
  \quad\Longrightarrow\quad
  \lambda_1(X)=\frac{(m-1)^2}{4}
  \quad\Longrightarrow\quad
  h_{\mathrm{iso}}(X)=m-1.
\]
Xiaodong Wang observed that the Ledrappier--Wang integral formula, together
with Gang Liu's equality argument, also proves rigidity under the curvature
assumption of this paper when $h_{\mathrm{vol}}(X)=m-1$
(see \cite{LedrappierWang2010}*{Theorem~1 and the proof of Theorem~2} and
\cite{Liu2011}*{proof of Theorem~1}).  Kuntao Jin and Bo Zhu proved rigidity
from the middle equality
(see \cite{JinZhu2026McKean}*{Theorem~1.1}).  This Cheeger rigidity theorem
is stronger: its hypothesis
$h_{\mathrm{iso}}(X)=m-1$ is the weakest of the three equality assumptions.
The comparison alone gives neither equality in $\lambda_1(X)$ nor equality in
$h_{\mathrm{vol}}(X)$ from $h_{\mathrm{iso}}(X)=m-1$.

\subsection*{Proof strategy}

Section~\ref{sec:busemann-yau-defects} proves Busemann comparison and the
Yau defect identity
\begin{align*}
  \area(\partial\Omega)-(m-1)\vol(\Omega)
  ={}&\int_{\partial\Omega}
  \bigl(1-\langle\nabla B_\xi,\nu\rangle\bigr)\,dA\\
  &+\int_\Omega\bigl(\Delta B_\xi-(m-1)\bigr)\,dV.
\end{align*}
For an isoperimetric minimizing sequence, both nonnegative terms vanish
after division by $\vol(\Omega_j)$.

Section~\ref{sec:suspension-localization} pushes the normalized volume
measures to the compact suspension
$Z=(X\times\partial_\infty X)/\Gamma$.  A subsequence converges weakly to
a probability measure $\mu$, and the interior term gives
\[
  \supp\mu\subseteq
  \{[x,\xi]\in Z:\Delta B_\xi(x)=m-1\}.
\]
Section~\ref{sec:strong-stable-divergence} identifies horospheres with
strong stable leaves.  The boundary term, the leafwise Gauss--Green formula,
and the tangency-preserving approximation from
Appendix~\ref{sec:horosphere-tangent-approximation} give
\[
  \int_Z\divss Y\,d\mu=0
\]
for every continuous strong stable vector field $Y$ with continuous
leafwise covariant derivative.

Section~\ref{sec:disintegration-full-support} proves that the conditional
measures of $\mu$ are multiples of Riemannian plaque volume.  Thus
$\supp\mu$ is saturated by strong stable leaves.  Eberlein's minimality
theorem implies that every such leaf is dense
(see \cite{Eberlein1973}*{Theorem~6.1}).  Hence $\supp\mu=Z$, and
\[
  \Delta B_\xi(x)=m-1
  \qquad\text{for every }(x,\xi)\in X\times\partial_\infty X.
\]

Section~\ref{sec:proof-main-theorem} uses equality in Busemann comparison to
obtain
\[
  \nabla^2B_\xi=g_X-dB_\xi\otimes dB_\xi,
  \qquad
  g_X=dt^2+e^{2t}g_H.
\]
The curvature bounds force $g_H$ to be flat, so this warped product is
$\Hh^m(-1)$.  Section~\ref{sec:p-comparison} proves the lower bound for
the variational bottom of the $p$-Laplacian and characterizes equality.

\section{Busemann comparison and the Yau defect}
\label{sec:busemann-yau-defects}

Only the curvature bounds are used in this section; no cocompact group
action is needed.  Let
$(X^m,g_X)$, $m\geq2$, be a complete simply connected Riemannian
manifold, and assume that, for some $A\geq1$,
\begin{equation}\label{eq:pinched-curvature}
	-A^2\leq\secg_{g_X}\leq-1.
\end{equation}
Thus $X$ is a pinched negatively curved Hadamard manifold.  The standard
visual-boundary and Busemann facts are taken from
\cite{EO73}*{Sections~1--3} and \cite{HIH77}*{Sections~2--3}.

\subsection{Visual boundary and Busemann comparison}

A \emph{geodesic ray} is a unit-speed geodesic
\(c\colon[0,\infty)\to X\); thus
\[
d_X\bigl(c(s),c(t)\bigr)=|s-t|
\quad (s,t\geq0).
\]
Two geodesic rays \(c_1\) and \(c_2\) are called \emph{asymptotic}, written
\(c_1\sim c_2\), if
\[
\sup_{t\geq0}d_X\bigl(c_1(t),c_2(t)\bigr)<\infty.
\]
For unit-speed rays, this is equivalent to requiring their images to have
finite Hausdorff distance.  The triangle inequality shows that \(\sim\) is
an equivalence relation.  The \emph{visual boundary} of \(X\) is
\[
\partial_\infty X
:=
\{\text{geodesic rays in \(X\)}\}/\mathord{\sim}.
\]
The class \([c]\), also denoted by \(c(+\infty)\), is the endpoint of \(c\)
at infinity.

Fix a base point \(o\in X\) and set
\[
S_oX:=\{v\in T_oX:|v|_{g_X}=1\},
\quad
c_v(t):=\exp_o(tv).
\]
Every \(\xi\in\partial_\infty X\) has a unique representative
\(c_{o,\xi}\) issuing from \(o\).  Existence follows by taking a locally
uniform limit of the minimizing segments from \(o\) to points tending to
\(\xi\).  For uniqueness, if \(\gamma_1,\gamma_2\) are asymptotic rays
issuing from \(o\), then
\[
f(t):=d_X\bigl(\gamma_1(t),\gamma_2(t)\bigr)
\]
is bounded and convex with \(f(0)=0\).  Thus
\(f(s)\leq (s/t)f(t)\) for \(0<s<t\); letting \(t\to\infty\) gives
\(\gamma_1=\gamma_2\).

The endpoint map
\[
\Phi_o\colon S_oX\longrightarrow\partial_\infty X,
\quad
\Phi_o(v):=[c_v],
\]
is therefore a bijection.  The \emph{cone topology} on \(\partial_\infty X\)
is the topology transported from \(S_oX\) by \(\Phi_o\).  Thus
\(\xi_j\to\xi\) precisely when
\[
\dot c_{o,\xi_j}(0)\longrightarrow\dot c_{o,\xi}(0)
\quad\text{in }S_oX.
\]
Equivalently, \(c_{o,\xi_j}\to c_{o,\xi}\) uniformly on every bounded time
interval.

For another base point \(o'\in X\), define the reanchoring map by
\[
R_{o,o'}
:=
\Phi_{o'}^{-1}\circ\Phi_o
\colon S_oX\longrightarrow S_{o'}X.
\]
The same compactness and convexity argument shows that \(R_{o,o'}\) and
\(R_{o',o}\) are continuous.  Hence \(R_{o,o'}\) is a homeomorphism, so the
cone topology is independent of the base point.

An isometry \(F\in\operatorname{Isom}(X)\) acts on the visual
boundary by
\[
F_\infty([c]):=[F\circ c].
\]
This is well defined because \(F\) preserves distances.  In the endpoint
coordinates it has the form
\[
F_\infty
=
\Phi_{F(o)}
\circ\left.dF_o\right|_{S_oX}
\circ\Phi_o^{-1},
\]
so \(F_\infty\) is a homeomorphism of \(\partial_\infty X\).

For \(\xi\in\partial_\infty X\), define the Busemann function normalized at
\(o\) by
\begin{equation}\label{eq:busemann-definition}
	B_{\xi,o}(x)
	:=\lim_{t\to\infty}
	\bigl(d_X(x,c_{o,\xi}(t))-t\bigr).
\end{equation}
For existence of the limit, set
\[
h_t(x):=d_X(x,c_{o,\xi}(t))-t.
\]
The triangle inequality shows that \(t\mapsto h_t(x)\) is nonincreasing and
that
\[
-d_X(x,o)\leq h_t(x)\leq d_X(x,o).
\]
Hence the limit in \eqref{eq:busemann-definition} exists, is finite, and
satisfies \(B_{\xi,o}(o)=0\).  Since the functions \(h_t\) are uniformly
\(1\)-Lipschitz, they converge uniformly on compact subsets, and
\(B_{\xi,o}\) is \(1\)-Lipschitz.

Changing the base point changes only the additive normalization: for
\(o'\in X\),
\[
B_{\xi,o'}(x)
=
B_{\xi,o}(x)-B_{\xi,o}(o').
\]
Define the Busemann cocycle by
\[
\beta_\xi(x,y)
:=
B_{\xi,o}(x)-B_{\xi,o}(y);
\]
by the base-point change identity, this definition is independent of \(o\).
It satisfies
\[
\beta_\xi(x,z)
=
\beta_\xi(x,y)+\beta_\xi(y,z),
\quad
|\beta_\xi(x,y)|\leq d_X(x,y).
\]
For \(F\in\operatorname{Isom}(X)\),
\[
B_{F_\infty(\xi),F(o)}(F(x))=B_{\xi,o}(x),
\quad
\beta_{F_\infty(\xi)}(F(x),F(y))=\beta_\xi(x,y).
\]
Since the gradient, Hessian, and Laplacian are unchanged by adding a
constant, the notation \(B_\xi\) will be used whenever only derivatives
are involved.

\begin{proposition}
	\label{prop:busemann-comparison}
	Let \((X^m,g_X)\), \(m\geq2\), be a complete simply connected
	Riemannian manifold, and assume that
	\begin{equation*}
	-A^2\leq\secg_{g_X}\leq-1
	\end{equation*}
	for some \(A\geq1\).
	Then
	\begin{enumerate}[label=\textup{(\roman*)}]
		\item For every \(\xi\in\partial_\infty X\), the Busemann function
		\(B_\xi\) is of class \(C^2\).  If \(v_{x,\xi}\in S_xX\) is the initial
		velocity of the unique ray from \(x\) to \(\xi\), then
		\begin{equation}\label{eq:unit-busemann-gradient}
			\nabla B_\xi(x)=-v_{x,\xi},
			\quad
			|\nabla B_\xi(x)|=1.
		\end{equation}
		
		\item For every \(\xi\in\partial_\infty X\),
		\begin{equation}\label{eq:busemann-hessian-comparison}
			g_X-dB_\xi\otimes dB_\xi
			\leq\nabla^2B_\xi
			\leq A(g_X-dB_\xi\otimes dB_\xi),
		\end{equation}
		and hence
		\begin{equation}\label{eq:busemann-laplacian-comparison}
			m-1\leq\Delta B_\xi\leq A(m-1).
		\end{equation}
		
		\item Let \(\operatorname{pr}_X\colon
		X\times\partial_\infty X\to X\) be the projection.  Then
		\[
		(x,\xi)\longmapsto\nabla B_\xi(x),
		\quad
		(x,\xi)\longmapsto\nabla^2B_\xi(x)
		\]
		define continuous sections of
		\(\operatorname{pr}_X^*TX\) and
		\(\operatorname{pr}_X^*\operatorname{Sym}^2(T^*X)\), respectively, and
		\[
		(x,\xi)\longmapsto\Delta B_\xi(x)
		\]
		is a continuous function on \(X\times\partial_\infty X\).
	\end{enumerate}
\end{proposition}

\begin{proof}
	Fix \(\xi\in\partial_\infty X\) and \(x\in X\), and write
	\(c=c_{x,\xi}\).  For \(R>0\), set
	\[
	b_R(y):=d_X(y,c(R))-R.
	\]
	On every compact subset of \(X\), the functions \(b_R\) are smooth for
	all sufficiently large \(R\) and converge to the Busemann function
	normalized by \(B_\xi(x)=0\).
	
	Stable-Jacobi-field convergence shows that the radial fields
	\(-\nabla b_R\) converge locally uniformly to
	\(v_{\,\cdot,\xi}\), and their covariant derivatives converge locally
	uniformly to the derivatives determined by the stable Jacobi fields along
	the rays to \(\xi\)
	(see \cite{HIH77}*{Lemma~2.2 and Proposition~3.1}).
	Equivalently, \(b_R\to B_\xi\) in \(C^2\) on compact subsets.  Hence
	\[
	B_\xi\in C^2(X),
	\quad
	\nabla B_\xi(x)=-v_{x,\xi},
	\quad
	|\nabla B_\xi|=1,
	\]
	which proves \textup{(i)}.  The convergence is uniform on compact
	families.
	
	For \textup{(ii)}, decompose
	\[
	w=a\,v_{x,\xi}+w^\perp,
	\quad
	w^\perp\perp v_{x,\xi}.
	\]
	Hessian comparison for the distance from \(c(R)\) gives
	\[
	\coth R\,|w^\perp|^2
	\leq \nabla^2b_R(x)(w,w)
	\leq A\coth(AR)\,|w^\perp|^2.
	\]
	Passing to the \(C^2\)-limit and using
	\(\nabla B_\xi=-v_{x,\xi}\) yields
	\[
	|w|^2-\langle w,\nabla B_\xi\rangle^2
	\leq\nabla^2B_\xi(w,w)
	\leq
	A\bigl(|w|^2-\langle w,\nabla B_\xi\rangle^2\bigr).
	\]
	This is \eqref{eq:busemann-hessian-comparison}.  Taking the trace in an
	orthonormal basis whose first vector is \(v_{x,\xi}\) gives the two
	Laplacian bounds in \eqref{eq:busemann-laplacian-comparison}.
	
	For \textup{(iii)}, the map
	\[
	(x,\xi)\longmapsto v_{x,\xi}
	\]
	is continuous by the cone topology, so the gradient formula gives the
	joint continuity of \(\nabla B_\xi(x)\).
	
	For the Hessian, fix \(T>0\) and define
	\[
	H^T_{x,\xi}
	:=
	\left.
	\nabla_y^2\bigl(d_X(y,c_{x,\xi}(T))-T\bigr)
	\right|_{y=x}.
	\]
	Because a Hadamard manifold has no cut locus and
	\((x,\xi)\mapsto c_{x,\xi}(T)\) is continuous, \(H^T\) is a continuous
	section of
	\(\operatorname{pr}_X^*\operatorname{Sym}^2(T^*X)\).
	
	Let \(J_w^T\) be the normal Jacobi field along
	\(c_{x,\xi}|_{[0,T]}\) with \(J_w^T(0)=w\) and \(J_w^T(T)=0\), and let
	\(J_w^s\) be the bounded normal Jacobi field with \(J_w^s(0)=w\).
	Set \(K=J_w^T-J_w^s\).  Then \(K\) is a normal Jacobi field satisfying
	\(K(0)=0\) and \(K(T)=-J_w^s(T)\).  Rauch comparison with the Euclidean
	Jacobi field having the same initial derivative gives
	\[
	  |K'(0)|\leq\frac{|K(T)|}{T}
	  =\frac{|J_w^s(T)|}{T}
	\qquad
	\text{(see \cite{HIH77}*{Lemma~2.2}).}
	\]
	Because \(\secg_{g_X}\leq0\), the function \(t\mapsto|J_w^s(t)|\) is
	convex.  It is bounded on \([0,\infty)\), so it is nonincreasing and
	\(|J_w^s(T)|\leq|w|\).  Consequently,
	\[
	  \bigl|(J_w^T)'(0)-(J_w^s)'(0)\bigr|
	  \leq\frac{|w|}{T}.
	\]
	For \(w,u\perp v_{x,\xi}\), the Jacobi-field formulas are
	\[
	  H^T_{x,\xi}(w,u)
	  =-\langle(J_w^T)'(0),u\rangle,
	  \qquad
	  \nabla^2B_\xi(x)(w,u)
	  =-\langle(J_w^s)'(0),u\rangle.
	\]
	Both Hessians vanish whenever one argument is \(v_{x,\xi}\).  Hence the
	preceding vector estimate gives the uniform operator-norm bound
	\[
	\bigl\|H^T_{x,\xi}-\nabla^2B_\xi(x)\bigr\|
	\leq\frac1T.
	\]
	The estimate is independent of \((x,\xi)\), so
	\(\nabla^2B_\xi(x)\) is the uniform limit of the continuous sections
	\(H^T\) and is jointly continuous.  Taking the trace proves the joint
	continuity of \(\Delta B_\xi(x)\).
\end{proof}

Part~\textup{(iii)} follows from the uniform finite-sphere approximation
of the Hessian.  The two-sided curvature bound is
essential: Busemann functions on a general Hadamard manifold need not be
$C^2$.  In constant sectional curvature $-1$, equality holds in the
comparison estimate:
\[
\nabla^2B_\xi=g_X-dB_\xi\otimes dB_\xi,
\quad
\Delta B_\xi=m-1.
\]
In particular,
\begin{equation}\label{eq:hessian-comparison}
  |\nabla B_\xi|=1,
  \qquad
  \nabla^2B_\xi\geq g_X-dB_\xi\otimes dB_\xi.
\end{equation}
Taking the trace gives
\[
  \Delta B_\xi-(m-1)\geq 0.
\]

\subsection{The Yau defect identity and isoperimetric minimizing sequences}
Yau's lower bound follows from a flux estimate.  Let $B\in C^2(X)$ satisfy
$|\nabla B|\leq1$ and
$\Delta B\geq c$.  For every relatively compact domain
$\Omega\Subset X$ with smooth boundary and outward unit normal $\nu$,
the Laplacian bound, the divergence theorem, and the gradient bound give
\[
  c\,\vol(\Omega)
  \leq \int_\Omega\Delta B\,dV
  =\int_{\partial\Omega}\langle\nabla B,\nu\rangle\,dA
  \leq \area(\partial\Omega).
\]
Dividing by $\vol(\Omega)$ and taking the infimum over $\Omega$ yields
$\hiso(X)\geq c$.  In particular, applying the estimate to a Busemann
function $B_\xi$, for which $|\nabla B_\xi|=1$ and
$\Delta B_\xi\geq m-1$, gives $\hiso(X)\geq m-1$.

For a domain $\Omega$, define its \emph{Yau excess} to be
\[
  \area(\partial\Omega)-(m-1)\vol(\Omega).
\]

\begin{lemma}
\label{lem:yau-defect}
Let $(X^m,g_X)$, $m\geq2$, be a Riemannian manifold without boundary,
and let $B\in C^2(X)$.  If $\Omega\Subset X$ is a domain with
smooth boundary and outward unit normal $\nu$, then
\begin{equation}\label{eq:yau-defect}
  \area(\partial\Omega)-(m-1)\vol(\Omega)
  =\int_{\partial\Omega}
  \bigl(1-\langle\nabla B,\nu\rangle\bigr)\,dA
  +\int_\Omega\bigl(\Delta B-(m-1)\bigr)\,dV.
\end{equation}
\end{lemma}

\begin{proof}
The divergence theorem gives
\[
  \int_\Omega\Delta B\,dV
  =\int_{\partial\Omega}\langle\nabla B,\nu\rangle\,dA.
\]
Adding and subtracting the flux gives
\[
\begin{array}{cl}
  & \displaystyle \area(\partial\Omega)-(m-1)\vol(\Omega) \\[4pt]
  ={} & \displaystyle \int_{\partial\Omega}1\,dA
      -(m-1)\int_\Omega1\,dV \\[4pt]
  ={} & \displaystyle \int_{\partial\Omega}
      \bigl(1-\langle\nabla B,\nu\rangle\bigr)\,dA
      +\int_\Omega\bigl(\Delta B-(m-1)\bigr)\,dV.
\end{array}
\]
\end{proof}

For $B=B_\xi$, the Yau defect identity~\eqref{eq:yau-defect} separates the
two errors in the flux estimate.  The first is the \emph{boundary defect}
and the second the \emph{interior defect}.
By \eqref{eq:hessian-comparison}, $|\nabla B_\xi|=1$ and
$\Delta B_\xi\geq m-1$.  Together with $|\nu|=1$, these facts imply
\[
  1-\langle\nabla B_\xi,\nu\rangle\geq0
  \quad\text{on }\partial\Omega,
  \qquad
  \Delta B_\xi-(m-1)\geq0
  \quad\text{on }\Omega.
\]
The first inequality is Cauchy--Schwarz, so both defects are
nonnegative.  If
$\hiso(X)=m-1$, dividing \eqref{eq:yau-defect} by
$\vol(\Omega_j)$ along an isoperimetric minimizing sequence shows that the
sum of the two normalized defects tends to zero.  Each normalized defect
must therefore tend to zero separately.  The boundary defect also controls
the misalignment of the boundary normal with $\nabla B_\xi$.

\begin{proposition}
\label{prop:yau-defects}
Let $(X^m,g_X)$, $m\geq2$, be a complete simply connected Riemannian
manifold satisfying $-A^2\leq\secg_{g_X}\leq-1$ for some $A\geq1$.
If $\hiso(X)=m-1$, then there exists a sequence of relatively compact domains
$\Omega_j\Subset X$ with smooth boundary such that
\begin{equation}\label{eq:isoperimetric-sequence}
  \frac{\area(\partial\Omega_j)}{\vol(\Omega_j)}
  \longrightarrow m-1.
\end{equation}
Let $\{\Omega_j\}$ be any sequence satisfying
\eqref{eq:isoperimetric-sequence}.  Fix $\xi_0\in\partial_\infty X$, set
\[
  B=B_{\xi_0},
  \qquad
  N=\nabla B.
\]
If $\nu_j$ denotes the outward unit normal along $\partial\Omega_j$,
then both limits below hold:
\begin{enumerate}[label=\textup{(\roman*)}]
\item The interior defect vanishes:
\begin{equation}\label{eq:interior-defect}
  \frac{1}{\vol(\Omega_j)}\int_{\Omega_j}
  \bigl(\Delta B-(m-1)\bigr)\,dV
  \longrightarrow0.
\end{equation}

\item The boundary defect vanishes:
\begin{equation}\label{eq:normal-alignment}
  \frac{1}{\vol(\Omega_j)}
  \int_{\partial\Omega_j}|\nu_j-N|^2\,dA\longrightarrow 0.
\end{equation}
\end{enumerate}
\end{proposition}

\begin{proof}
By the definition of the Cheeger constant, for each $j\geq1$ there is a
relatively compact domain $\Omega_j\Subset X$ with smooth boundary such
that
\[
  m-1=\hiso(X)
  \leq
  \frac{\area(\partial\Omega_j)}{\vol(\Omega_j)}
  <m-1+\frac1j.
\]
Thus \eqref{eq:isoperimetric-sequence} holds.

Now fix any sequence satisfying \eqref{eq:isoperimetric-sequence} and any
$\xi_0\in\partial_\infty X$.  With $B$, $N$, and $\nu_j$ as in the
statement, define
\begin{align*}
  E_j
  &:=\frac{\area(\partial\Omega_j)}{\vol(\Omega_j)}-(m-1),\\
  I_j
  &:=\frac{1}{\vol(\Omega_j)}\int_{\Omega_j}
  \bigl(\Delta B-(m-1)\bigr)\,dV,\\
  J_j
  &:=\frac{1}{\vol(\Omega_j)}\int_{\partial\Omega_j}
  \bigl(1-\langle N,\nu_j\rangle\bigr)\,dA.
\end{align*}
By \eqref{eq:isoperimetric-sequence}, $E_j\to0$.  Proposition~\ref{prop:busemann-comparison}
gives $|N|=1$ and $\Delta B\geq m-1$, so $I_j,J_j\geq0$.
Lemma~\ref{lem:yau-defect}, divided by $\vol(\Omega_j)$, gives
\[
  E_j=I_j+J_j.
\]
Consequently,
\[
  0\leq I_j\leq E_j,
  \qquad
  0\leq J_j\leq E_j.
\]
Thus $I_j\to0$ and $J_j\to0$.  The first limit is
\eqref{eq:interior-defect}.

Since $|N|=|\nu_j|=1$ on $\partial\Omega_j$,
\[
  |\nu_j-N|^2
  =|\nu_j|^2+|N|^2-2\langle N,\nu_j\rangle
  =2\bigl(1-\langle N,\nu_j\rangle\bigr).
\]
Therefore
\[
  \frac{1}{\vol(\Omega_j)}
  \int_{\partial\Omega_j}|\nu_j-N|^2\,dA
  =2J_j\longrightarrow0.
\]
\end{proof}

\section{Horospherical suspension and measure localization}
\label{sec:suspension-localization}

Fix an isoperimetric minimizing sequence $\{\Omega_j\}$ as in
Proposition~\ref{prop:yau-defects} and a point
$\xi_0\in\partial_\infty X$.  Normalized volume measures on the noncompact
space $X$ may lose mass at infinity.  After passage to the compact
horospherical suspension, a subsequence converges weakly.
Proposition~\ref{prop:yau-defects}(i) shows that every such limit is
supported on the equality set of Busemann Laplacian comparison.

Recall that the deck group $\Gamma$ acts diagonally on
$X\times\partial_\infty X$ by
\[
  \gamma\cdot(x,\xi)=(\gamma x,\gamma\xi).
\]
Set $Z:=(X\times\partial_\infty X)/\Gamma$.  This is the visual-boundary
model of the horospherical suspension of the
regular cover $X\to M$ (see
\cite{Ledrappier2010}*{Section~1}; see
\cite{LedrappierWang2010}*{Section~2}).

The compactness of $Z$ follows from its identification with $SM$.  For
$(x,\xi)\in X\times\partial_\infty X$, let $c_{x,\xi}$ be the unique
unit-speed geodesic ray from $x$ to $\xi$, and set
\[
  v_{x,\xi}:=\dot c_{x,\xi}(0)\in S_xX.
\]
Define
\[
  \widetilde\Phi:X\times\partial_\infty X\longrightarrow SX,
  \qquad
  \widetilde\Phi(x,\xi)=v_{x,\xi}.
\]
If $\pi_X:SX\to X$ denotes the footpoint projection and $v_+$ the
forward endpoint of the geodesic with initial vector $v$, then
\[
  \widetilde\Phi^{-1}(v)=(\pi_X(v),v_+).
\]
The cone-topology results of
Section~\ref{sec:busemann-yau-defects} show that $\widetilde\Phi$ and
$\widetilde\Phi^{-1}$ are continuous
(see \cite{EO73}*{Sections~1--2}).  Thus
$\widetilde\Phi$ is a homeomorphism.

For $\gamma\in\Gamma$, the curve $\gamma\circ c_{x,\xi}$ is a unit-speed
geodesic ray from $\gamma x$ to $\gamma\xi$.  Uniqueness gives
$c_{\gamma x,\gamma\xi}=\gamma\circ c_{x,\xi}$.  Differentiating at
$t=0$ gives
\[
  \widetilde\Phi(\gamma x,\gamma\xi)
  =(d\gamma_x)(\widetilde\Phi(x,\xi)).
\]
Hence $\widetilde\Phi$ is $\Gamma$-equivariant for the differential
action on $SX$.  Since $SX/\Gamma$ is canonically identified with $SM$,
it induces the homeomorphism
\begin{equation}\label{eq:Z-SM-homeo}
  \overline\Phi:Z\longrightarrow SM,
  \qquad
  \overline\Phi([x,\xi])=d\pi_x(v_{x,\xi}),
\end{equation}
where $\pi:(X,g_X)\to(M,g)$ is the universal Riemannian covering.
Because $M$ is closed, $SM$ is compact, and therefore so is $Z$.

The Busemann--Laplacian term in the interior defect descends to $Z$.  On
$X\times\partial_\infty X$, set
\[
  \widetilde D(x,\xi):=\Delta B_\xi(x)-(m-1).
\]
Proposition~\ref{prop:busemann-comparison}(ii)--(iii) shows that
$\widetilde D$ is continuous and nonnegative.  The equivariance of
Busemann functions under isometries, established in
Section~\ref{sec:busemann-yau-defects}, implies
\[
  \Delta B_{\gamma\xi}(\gamma x)=\Delta B_\xi(x)
  \qquad(\gamma\in\Gamma).
\]
Thus $\widetilde D$ is $\Gamma$-invariant and descends to a continuous
function $D:Z\to[0,\infty)$, given by
\begin{equation}\label{eq:D-on-Z}
  D([x,\xi])=\Delta B_\xi(x)-(m-1).
\end{equation}

Set $v_j:=\vol(\Omega_j)$, and define
$\iota_{\xi_0}:X\to Z$ by $\iota_{\xi_0}(x)=[x,\xi_0]$.
Push the normalized volume measure of $\Omega_j$ forward by
$\iota_{\xi_0}$:
\begin{equation}\label{eq:mu-j}
  \mu_j=(\iota_{\xi_0})_*
  \left(\frac{\mathbf{1}_{\Omega_j}}{v_j}\,dV\right).
\end{equation}
Since $v_j=\vol(\Omega_j)$,
\[
  \mu_j(Z)=\frac{\vol(\Omega_j)}{v_j}=1,
\]
so each $\mu_j$ is a Borel probability measure.  The space of Borel
probability measures on the compact metric space $Z$ is sequentially
compact in the topology of weak convergence
(see \cite{Bogachev2007}*{Theorem~8.6.7}).  After passing to a
subsequence,
\[
  \mu_j\rightharpoonup\mu
\]
for some Borel probability measure $\mu$ on $Z$.

\begin{proposition}
\label{prop:limit-support}
Every weak limit $\mu$ of the measures $\{\mu_j\}$ defined in
\eqref{eq:mu-j} satisfies
\[
  \supp\mu\subseteq D^{-1}(0).
\]
\end{proposition}

\begin{proof}
Let $\mu_j\rightharpoonup\mu$ along a subsequence.  Equations
\eqref{eq:mu-j} and \eqref{eq:D-on-Z} give
\begin{align*}
  \int_ZD\,d\mu_j
  &=\frac1{v_j}\int_{\Omega_j}
  D\bigl(\iota_{\xi_0}(x)\bigr)\,dV(x)\\
  &=\frac{1}{\vol(\Omega_j)}\int_{\Omega_j}
  \bigl(\Delta B_{\xi_0}(x)-(m-1)\bigr)\,dV(x).
\end{align*}
By Proposition~\ref{prop:yau-defects}(i), the last expression tends to
zero.  Therefore
\[
  \int_ZD\,d\mu_j\longrightarrow0.
\]
Since $D$ is continuous on $Z$, weak convergence yields
\[
  \int_ZD\,d\mu
  =\lim_{j\to\infty}\int_ZD\,d\mu_j
  =0.
\]

Suppose that $D(z)>0$ for some $z\in\supp\mu$.  By continuity, there are
an open neighborhood $U$ of $z$ and $\varepsilon>0$ such that
$D\geq\varepsilon$ on $U$.  The definition of the support gives
$\mu(U)>0$.  Since $D\geq0$ on $Z$,
\[
  0=\int_ZD\,d\mu
  \geq\int_UD\,d\mu
  \geq\varepsilon\mu(U)>0,
\]
a contradiction.  Hence $D=0$ on $\supp\mu$.
\end{proof}

\section{Strong stable leaves and leafwise divergence}
\label{sec:strong-stable-divergence}
Section~\ref{sec:suspension-localization} yields a probability measure
$\mu$ on the compact suspension $Z$ with $\supp\mu\subseteq D^{-1}(0)$.
Boundary-normal convergence in
\eqref{eq:normal-alignment} implies that $\mu$ annihilates divergence
along strong stable leaves.  In the boundary coordinates
$X\times\partial_\infty X$, a weak stable leaf has fixed forward endpoint,
and its strong stable leaves are the level sets of the corresponding
Busemann function.  These horospheres coincide with the dynamical strong
stable sets, and they satisfy the leafwise Gauss--Green formula established
below.

\subsection{Horospherical geometry}
The strong-stability estimates require connected horospheres and a uniform
bound for the tangential derivative of their unit normal.  Both statements
follow from Busemann comparison.

\begin{lemma}
\label{lem:horospherical-geometry}
Let $(X^m,g_X)$, $m\geq2$, be a complete simply connected Riemannian
manifold satisfying $-A^2\leq\secg_{g_X}\leq-1$ for some $A\geq1$.
For $\xi\in\partial_\infty X$, let $B_\xi$ be a Busemann function
centered at $\xi$, and set
\[
  H_{\xi,c}:=B_\xi^{-1}(c)
  \quad(c\in\R).
\]
Then:
\begin{enumerate}[label=\textup{(\roman*)}]
\item The horosphere $H_{\xi,c}$ is a connected embedded $C^2$
hypersurface with unit normal $\nabla B_\xi$.  For every
$x\in H_{\xi,c}$,
\begin{equation}\label{eq:tangent-horosphere}
  T_xH_{\xi,c}
  =\ker dB_\xi|_x
  =\bigl(\nabla B_\xi(x)\bigr)^\perp.
\end{equation}

\item With the sign convention
\begin{equation}\label{eq:horosphere-shape-operator}
  S_{\xi,x}:T_xH_{\xi,c}\longrightarrow T_xH_{\xi,c},
  \qquad
  S_{\xi,x}(V)=\nabla_V\nabla B_\xi,
\end{equation}
the shape operator is self-adjoint and satisfies
\begin{equation}\label{eq:horosphere-shape-bound}
\begin{aligned}
  \operatorname{Id}&\leq S_{\xi,x}\leq A\operatorname{Id}
  &&\text{as quadratic forms},\\
  |S_{\xi,x}V|&\leq A|V|
  &&\text{for }V\in T_xH_{\xi,c}.
\end{aligned}
\end{equation}
\end{enumerate}
\end{lemma}

\begin{proof}
Proposition~\ref{prop:busemann-comparison}(i) gives
$B_\xi\in C^2(X)$ and $|\nabla B_\xi|=1$.  Hence $dB_\xi$ never vanishes,
so every $c\in\R$ is a regular value.  The regular-level-set theorem gives
an embedded $C^2$ hypersurface with
\[
  T_xH_{\xi,c}=\ker dB_\xi|_x.
\]
Since $dB_\xi(V)=\langle\nabla B_\xi,V\rangle$, this kernel is
$\bigl(\nabla B_\xi(x)\bigr)^\perp$, and $\nabla B_\xi$ is a unit
normal.

For connectedness, consider the $C^1$ vector field $\nabla B_\xi$.
It is complete: its integral curves have unit
speed, so an integral curve with a finite maximal endpoint remains in a
closed bounded ball.  This ball is compact by Hopf--Rinow, and the local
existence theorem for ordinary differential equations extends the curve
past that endpoint, a contradiction.  Let $\psi_t$ be the global flow of
$\nabla B_\xi$.  Along an integral curve,
\[
  \frac{d}{dt}B_\xi(\psi_t(x))
  =dB_\xi(\nabla B_\xi)
  =|\nabla B_\xi|^2
  =1.
\]
Thus $B_\xi(\psi_t(x))=B_\xi(x)+t$, and every flow line meets
$H_{\xi,c}$ exactly once.  The map
\[
  F_c:H_{\xi,c}\times\R\longrightarrow X,
  \qquad F_c(y,t)=\psi_t(y),
\]
is therefore a $C^1$ diffeomorphism with inverse
\[
  x\longmapsto
  \bigl(\psi_{c-B_\xi(x)}(x),B_\xi(x)-c\bigr).
\]
Since $X$ is connected, so is $H_{\xi,c}\times\R$.  Its projection onto
$H_{\xi,c}$ is continuous and surjective; hence $H_{\xi,c}$ is
connected, proving \textup{(i)}.

If $V\in T_xH_{\xi,c}$, then
$\langle\nabla_V\nabla B_\xi,\nabla B_\xi\rangle
=\frac12V(|\nabla B_\xi|^2)=0$; hence
$\nabla_V\nabla B_\xi$ lies in $T_xH_{\xi,c}$, so
\eqref{eq:horosphere-shape-operator} is well defined.  For
$V,W\in T_xH_{\xi,c}$,
\[
  \langle S_{\xi,x}V,W\rangle
  =\nabla^2B_\xi(V,W).
\]
Thus $S_{\xi,x}$ is self-adjoint.  Since $dB_\xi(V)=0$, the Hessian
comparison in Proposition~\ref{prop:busemann-comparison}(ii) gives
\[
  |V|^2\leq\langle S_{\xi,x}V,V\rangle\leq A|V|^2.
\]
Hence the spectrum of the self-adjoint operator $S_{\xi,x}$ lies in
$[1,A]$, so
$\lVert S_{\xi,x}\rVert_{\mathrm{op}}\leq A$.  Hence
$|S_{\xi,x}V|\leq A|V|$ for every $V\in T_xH_{\xi,c}$, completing
\textup{(ii)}.
\end{proof}

\subsection{Sasaki geometry and strong stability}
Strong stability requires convergence in the unit tangent bundle, not
only convergence of the footpoints in $X$.  The Sasaki metric measures the
motion of both the footpoint and the tangent direction.

Let
\[
  SX:=\{v\in TX:|v|=1\},
  \qquad
  \pi_X:SX\longrightarrow X,
\]
where $\pi_X(v)=x$ for $v\in T_xX$.  The fiber over $x\in X$ is
\[
  S_xX:=\pi_X^{-1}(x)=\{v\in T_xX:|v|=1\}.
\]
Fix $v\in S_xX$.  For $U\in T_vSX$, choose a $C^1$ curve
$\sigma:(-\varepsilon,\varepsilon)\to SX$ with
$\sigma(0)=v$ and $\dot\sigma(0)=U$, and put
$c=\pi_X\circ\sigma$.  The connection map at $v$ is
\[
  K_v(U):=
  \left.\frac{D}{ds}\sigma(s)\right|_{s=0}\in T_xX.
\]
The value of $K_v(U)$ is independent of the choice of $\sigma$.  Set
\[
  U^{\mathrm h}:=d\pi_X(U)=\dot c(0),
  \qquad
  U^{\mathrm v}:=K_v(U).
\]
Since $|\sigma(s)|=1$,
\[
  \langle U^{\mathrm v},v\rangle
  =\left.\frac12\frac{d}{ds}|\sigma(s)|^2\right|_{s=0}
  =0.
\]
Thus $U^{\mathrm v}\in v^\perp$, and the Levi--Civita connection gives the
linear isomorphism
\[
  T_vSX\longrightarrow T_xX\oplus v^\perp,
  \qquad
  U\longmapsto(U^{\mathrm h},U^{\mathrm v}).
\]

Under this identification, the Sasaki metric is
\[
  g_{\mathrm{Sasaki},v}(U,W)
  :=\langle U^{\mathrm h},W^{\mathrm h}\rangle_x
    +\langle U^{\mathrm v},W^{\mathrm v}\rangle_x
\]
(see \cite{Sasaki1958}*{Section~3, formula~(3.2)}).  Hence
\begin{equation}\label{eq:sasaki-metric}
  |U|_{\mathrm{Sasaki}}^2
  =|U^{\mathrm h}|^2+|U^{\mathrm v}|^2.
\end{equation}

Let $\alpha:[a,b]\to SX$ be piecewise $C^1$, and set
$c=\pi_X\circ\alpha$.  At each differentiability point, $\alpha(s)$ is
a unit vector field along $c(s)$, and \eqref{eq:sasaki-metric} gives
\[
  |\dot\alpha(s)|_{\mathrm{Sasaki}}^2
  =|\dot c(s)|^2+
    \left|\frac{D}{ds}\alpha(s)\right|^2.
\]
Its Sasaki length is therefore
\[
  L_{\mathrm{Sasaki}}(\alpha)
  =\int_a^b
  \left(
    |\dot c(s)|^2+
    \left|\frac{D}{ds}\alpha(s)\right|^2
  \right)^{1/2}ds.
\]
In particular,
\[
  L_X(c)=\int_a^b|\dot c(s)|\,ds
  \leq L_{\mathrm{Sasaki}}(\alpha).
\]
Let $d_{SX}$ be the Riemannian distance induced by the Sasaki metric.
Taking the infimum over curves $\alpha$ joining $v$ to $w$ gives
\begin{equation}\label{eq:footpoint-one-lipschitz}
  d_X\bigl(\pi_X(v),\pi_X(w)\bigr)
  \leq d_{SX}(v,w),
  \qquad v,w\in SX.
\end{equation}
Thus $\pi_X:(SX,d_{SX})\to(X,d_X)$ is $1$-Lipschitz.

Every $\gamma\in\Gamma$ preserves the Levi--Civita connection, so
$d\gamma:SX\to SX$ is an isometry of the Sasaki metric.  Equip
$SM=SX/\Gamma$ with the quotient Sasaki metric, and denote its distance
by $d_{SM}$.  If $v,w\in SX$ lift $\bar v,\bar w\in SM$, respectively,
then
\[
  d_{SM}(\bar v,\bar w)
  =\inf_{\gamma\in\Gamma}d_{SX}\bigl(v,d\gamma(w)\bigr).
\]

\begin{definition}\label{def:stable-sets}
For $v\in SX$, let $c_v:\R\to X$ be the geodesic determined by
$c_v(0)=\pi_X(v)$ and $\dot c_v(0)=v$.  The \emph{geodesic flow} on
$SX$ is
\[
  \varphi^t(v):=\dot c_v(t),
  \qquad
  \pi_X(\varphi^tv)=c_v(t).
\]
For every $\gamma\in\Gamma$, one has
$c_{d\gamma(v)}=\gamma\circ c_v$ and
$\varphi^t(d\gamma(v))=d\gamma(\varphi^t(v))$.  Thus $\varphi^t$
induces a flow on $SM$, denoted by the same symbol.
\begin{enumerate}[label=\textup{(\roman*)}]
\item The \emph{strong stable set} of $v$ is
\[
  W^{ss}_{SX}(v)
  =\left\{w\in SX:
  d_{SX}(\varphi^tv,\varphi^tw)\longrightarrow0
  \text{ as }t\to+\infty
  \right\}.
\]
\item The \emph{weak stable set} of $v$ is
\[
  W^{ws}_{SX}(v)
  =\left\{w\in SX:
  \sup_{t\geq0}d_X\bigl(\pi_X(\varphi^tv),
  \pi_X(\varphi^tw)\bigr)<\infty
  \right\}.
\]
\item For $\bar v\in SM$, its \emph{strong stable set} is
\[
  W^{ss}_{SM}(\bar v)
  =\left\{\bar w\in SM:
  d_{SM}(\varphi^t\bar v,\varphi^t\bar w)\longrightarrow0
  \text{ as }t\to+\infty
  \right\}.
\]
\end{enumerate}
\end{definition}

By \eqref{eq:footpoint-one-lipschitz}, strong stability implies weak
stability: the footpoint distance tends to zero and is therefore bounded
on $[0,\infty)$.  The distinction is the flow parameter.  Weak stability
allows a fixed displacement along the common asymptotic direction,
whereas strong stability requires the full tangent vectors to converge at
equal times.

\begin{example}
\label{ex:stable-models}
\emph{Euclidean space.}
Identify $S\R^m$ with $\R^m\times S^{m-1}$.  If
$v=(x,u)$ and $w=(y,u')$, then
$\varphi^t(x,u)=(x+tu,u)$ and
\[
  d_{S\R^m}\bigl(\varphi^tv,\varphi^tw\bigr)^2
  =|x-y+t(u-u')|^2+d_{S^{m-1}}(u,u')^2.
\]
Thus $W^{ws}_{S\R^m}(x,u)=\R^m\times\{u\}$ and
$W^{ss}_{S\R^m}(x,u)=\{(x,u)\}$.

\par\medskip
\noindent\emph{Hyperbolic space.}
In the upper half-space model,
\[
  \Hh^m=\{(z,r):z\in\R^{m-1},\ r>0\},
  \qquad
  g_{\Hh}=r^{-2}(dz^2+dr^2).
\]
Let
\[
  V_{z,r}:=r\partial_r\big|_{(z,r)}.
\]
This unit vector points toward $\infty\in\partial_\infty\Hh^m$, and
\[
  \varphi^t(V_{z,r})=V_{z,re^t}.
\]
The dilation $\delta_t(z,r)=(e^{-t}z,e^{-t}r)$ is an isometry of
$\Hh^m$.  Its differential preserves the Sasaki metric and gives
\begin{align*}
  &d_{S\Hh^m}\bigl(\varphi^t(V_{z,r}),
                    \varphi^t(V_{z',r'})\bigr) \\
  &\qquad
  =d_{S\Hh^m}\bigl(V_{e^{-t}z,r},V_{e^{-t}z',r'}\bigr)
  \longrightarrow d_{S\Hh^m}(V_{0,r},V_{0,r'})
  =\left|\log\frac{r'}{r}\right|.
\end{align*}
The footpoint estimate \eqref{eq:footpoint-one-lipschitz} gives the lower
bound
\[
  d_{S\Hh^m}(V_{0,r},V_{0,r'})
  \geq d_{\Hh^m}((0,r),(0,r'))
  =\left|\log\frac{r'}r\right|.
\]
Along the vertical geodesic segment from $(0,r)$ to $(0,r')$, the upward
unit tangent vector is parallel.  Its lift to $S\Hh^m$ therefore has the
same Sasaki length as the base segment and gives the reverse inequality.
Two hyperbolic rays remain a bounded distance apart exactly when they have
the same forward endpoint.  It follows that
\begin{align*}
  W^{ws}_{S\Hh^m}(V_{z,r})
  &=\{V_{z',r'}:z'\in\R^{m-1},\ r'>0\},\\
  W^{ss}_{S\Hh^m}(V_{z,r})
  &=\{V_{z',r}:z'\in\R^{m-1}\}.
\end{align*}

Normalize the Busemann function at $\infty$ by $B_\infty(0,1)=0$.
Then $B_\infty(z,r)=-\log r$.  For an arbitrary
$v=v_{x,\xi}\in S\Hh^m$, the preceding formulas become
\[
  W^{ws}_{S\Hh^m}(v)
  =\{v_{y,\xi}:y\in\Hh^m\},
  \qquad
  W^{ss}_{S\Hh^m}(v)
  =\{v_{y,\xi}:B_\xi(y)=B_\xi(x)\}.
\]
\end{example}

In the boundary coordinates of Section~3, fix
\[
  v=\widetilde\Phi(x,\xi)=v_{x,\xi},
  \qquad
  w=\widetilde\Phi(y,\eta)=v_{y,\eta}.
\]
The rays determined by $v$ and $w$ remain a bounded distance apart if and
only if they represent the same point of the visual boundary.  Therefore
\begin{equation}\label{eq:ws-upstairs}
  w\in W^{ws}_{SX}(v)
  \quad\Longleftrightarrow\quad
  \eta=\xi.
\end{equation}
The lifted weak stable set is thus
\[
  \tWws(x,\xi)
  :=\widetilde\Phi^{-1}\bigl(W^{ws}_{SX}(v)\bigr)
  =X\times\{\xi\}.
\]
In particular, $\tWws(x,\xi)$ depends only on $\xi$.

Within $X\times\{\xi\}$, define the equal-height set through $(x,\xi)$ by
\begin{equation}\label{eq:ss-upstairs}
  \tWss(x,\xi)
  :=\{(y,\xi)\in X\times\{\xi\}:B_\xi(y)=B_\xi(x)\}
  =H_{\xi,B_\xi(x)}\times\{\xi\}.
\end{equation}
This is the horosphere through $x$ with forward endpoint $\xi$.

For $p\in X$, use $c_{p,\xi}:\R\to X$ for the complete geodesic
extending the ray from $p$ to $\xi$.  Let $w=\widetilde\Phi(y,\xi)$.
By the definition of the geodesic flow,
\[
  \pi_X(\varphi^tv)=c_{x,\xi}(t),
  \qquad
  \pi_X(\varphi^tw)=c_{y,\xi}(t).
\]
Proposition~\ref{prop:busemann-comparison} gives
$\nabla B_\xi(c_{p,\xi}(t))=-\dot c_{p,\xi}(t)$.  Therefore, for every
$p\in X$,
\[
  \frac{d}{dt}B_\xi(c_{p,\xi}(t))
  =\left\langle\nabla B_\xi(c_{p,\xi}(t)),
    \dot c_{p,\xi}(t)\right\rangle
  =-1.
\]
Taking $p=x$ and $p=y$ and integrating from $0$ to $t$ gives
\begin{equation}\label{eq:busemann-height-flow}
  B_\xi\bigl(\pi_X(\varphi^tv)\bigr)=B_\xi(x)-t,
  \qquad
  B_\xi\bigl(\pi_X(\varphi^tw)\bigr)=B_\xi(y)-t.
\end{equation}
The \emph{flow line} through $(y,\xi)$ is the orbit of the geodesic flow
in the boundary coordinates:
\[
  \left\{\widetilde\Phi^{-1}(\varphi^tw):t\in\R\right\}
  =\left\{(c_{y,\xi}(t),\xi):t\in\R\right\}.
\]
Set $a=B_\xi(y)-B_\xi(x)$.  Then
\[
  B_\xi\bigl(\pi_X(\varphi^aw)\bigr)
  =B_\xi(y)-a=B_\xi(x).
\]
Since $B_\xi(c_{y,\xi}(t))=B_\xi(y)-t$, this value of $a$ is unique.
Thus every flow line in $X\times\{\xi\}$ meets $\tWss(x,\xi)$ exactly
once.  The map
\[
  \Psi_{x,\xi}:\tWss(x,\xi)\times\R
  \longrightarrow\tWws(x,\xi),
  \qquad
  ((p,\xi),s)\longmapsto
  \bigl(\pi_X(\varphi^sv_{p,\xi}),\xi\bigr),
\]
has inverse
\[
  (y,\xi)\longmapsto
  \left(
    \bigl(c_{y,\xi}(B_\xi(y)-B_\xi(x)),\xi\bigr),
    B_\xi(x)-B_\xi(y)
  \right).
\]
Lemma~\ref{lem:horospherical-geometry} shows that the flow of
$\nabla B_\xi$ is complete and $C^1$; hence $\Psi_{x,\xi}$ is a $C^1$
diffeomorphism.  Since $H_{\xi,B_\xi(x)}$ is a hypersurface
in $X^m$,
\[
  \dim\tWws(x,\xi)=m,
  \qquad
  \dim\tWss(x,\xi)=m-1.
\]

To pass from $X\times\partial_\infty X$ to the quotient $Z$, let
$\gamma\in\Gamma$.  The invariance of the Busemann cocycle under
isometries shows that, for some constant $C(\gamma,\xi)$ independent of
$y$,
\begin{equation}\label{eq:stable-equivariance}
  B_{\gamma\xi}(\gamma y)
  =B_\xi(y)+C(\gamma,\xi).
\end{equation}
For $U\in T_yX$, differentiating \eqref{eq:stable-equivariance} gives
\[
\begin{aligned}
  \bigl\langle \nabla B_{\gamma\xi}(\gamma y),d\gamma_yU\bigr\rangle
  &=dB_{\gamma\xi}|_{\gamma y}(d\gamma_yU)\\
  &=d(B_{\gamma\xi}\circ\gamma)|_y(U)\\
  &=dB_\xi|_y(U)\\
  &=\bigl\langle \nabla B_\xi(y),U\bigr\rangle.
\end{aligned}
\]
Since $\gamma$ is an isometry,
\[
  \bigl\langle \nabla B_\xi(y),U\bigr\rangle
  =\bigl\langle d\gamma_y\nabla B_\xi(y),d\gamma_yU\bigr\rangle.
\]
The preceding identities hold for every $U\in T_yX$, and $d\gamma_y$ is
surjective.  Therefore
\[
  \nabla B_{\gamma\xi}(\gamma y)
  =d\gamma_y\nabla B_\xi(y).
\]
Equation~\eqref{eq:stable-equivariance} also implies that
\[
  \gamma(H_{\xi,r})
  =H_{\gamma\xi,r+C(\gamma,\xi)}.
\]
It follows from the definitions of \(\tWws\) and \(\tWss\) that the
diagonal action satisfies
\[
  \gamma\bigl(\tWws(x,\xi)\bigr)
  =\tWws(\gamma x,\gamma\xi),
  \qquad
  \gamma\bigl(\tWss(x,\xi)\bigr)
  =\tWss(\gamma x,\gamma\xi).
\]
Let $p:X\times\partial_\infty X\to Z$ be the quotient map, and set
\[
  \Wws([x,\xi]):=p\bigl(\tWws(x,\xi)\bigr),
  \qquad
  \Wss([x,\xi]):=p\bigl(\tWss(x,\xi)\bigr).
\]
These definitions do not depend on the representative of $[x,\xi]$.
If $(x',\xi')=\gamma(x,\xi)$, then equivariance and
$p\circ\gamma=p$ give
\[
  p\bigl(\tWws(x',\xi')\bigr)
  =p\bigl(\tWws(x,\xi)\bigr),
  \qquad
  p\bigl(\tWss(x',\xi')\bigr)
  =p\bigl(\tWss(x,\xi)\bigr).
\]
Thus $\Wws$ and $\Wss$ define, respectively, the weak stable and
horospherical sets on $Z$.

The definition of $\tWss(x,\xi)$ fixes both the forward endpoint $\xi$
and the Busemann height $B_\xi(x)$.  See
\cite{Eberlein1973}*{Section~3} for the relation between equal-height
horospheres and strong stable sets, and
\cite{HirschPughShub1977}*{Theorem~4.1 and Section~5} for general stable
manifold theory.  The relevant horospherical contraction estimate is the
case $a=1$ and $b=A$ of
\cite{HIH77}*{Proposition~3.2}.

\begin{proposition}
\label{prop:stable-leaf-identification}
Let $(M^m,g)$, $m\geq2$, be a closed connected Riemannian manifold
satisfying $-A^2\leq\secg_g\leq-1$ for some $A\geq1$.  Let
$(X,g_X)$ be its universal Riemannian cover, let $\Gamma$ be the deck
group, and set
$Z=(X\times\partial_\infty X)/\Gamma$.  Let
$\overline\Phi:Z\to SM$ be the homeomorphism induced by
$\widetilde\Phi(x,\xi)=v_{x,\xi}$.  Then, for
$(x,\xi)\in X\times\partial_\infty X$ and
$v=\widetilde\Phi(x,\xi)$,
\begin{equation}\label{eq:ss-identification-upstairs}
  \widetilde\Phi\bigl(\tWss(x,\xi)\bigr)
  =W^{ss}_{SX}(v).
\end{equation}
For every $z\in Z$,
\[
  \overline\Phi\bigl(\Wss(z)\bigr)
  =W^{ss}_{SM}\bigl(\overline\Phi(z)\bigr).
\]
\end{proposition}

\begin{proof}
Fix $(x,\xi)\in X\times\partial_\infty X$ and set
$v=\widetilde\Phi(x,\xi)$.

\smallskip
\noindent\emph{Step 1: equal height implies strong stability on $SX$.}
Let $(y,\xi)\in\tWss(x,\xi)$ and put
$w=\widetilde\Phi(y,\xi)$.  By definition,
$B_\xi(y)=B_\xi(x)$.  Lemma~\ref{lem:horospherical-geometry} shows that
$H_{\xi,B_\xi(x)}$ is a connected $C^2$ hypersurface, so there is a
piecewise $C^2$ curve
\[
  \alpha:[0,1]\longrightarrow H_{\xi,B_\xi(x)},
  \qquad
  \alpha(0)=x,
  \quad
  \alpha(1)=y.
\]
For $s\in[0,1]$ and $t\geq0$, set
\[
  F(s,t):=c_{\alpha(s),\xi}(t),
  \qquad
  J(s,t):=\partial_sF(s,t).
\]
For fixed $s$, the curve $t\mapsto F(s,t)$ is the unit-speed ray from
$\alpha(s)$ toward $\xi$.  Proposition~\ref{prop:busemann-comparison}(i)
gives
\[
  \partial_tF=-\nabla B_\xi(F).
\]
Because $B_\xi(\alpha(s))=B_\xi(x)$ and $B_\xi$ decreases at unit speed
along every such ray,
\[
  B_\xi(F(s,t))=B_\xi(x)-t.
\]
The right-hand side is independent of $s$, hence
$dB_\xi(J)=0$.  Thus $J(s,t)$ is tangent to the horosphere
$H_{\xi,B_\xi(x)-t}$.  Moreover, the coordinate vector fields of $F$
commute, and the Levi--Civita connection is torsion-free.  Therefore
\[
  \nabla_{\partial_tF}J
  =\nabla_J\partial_tF
  =-\nabla_J\nabla B_\xi.
\]
Since $J$ is tangent to the horosphere through $F(s,t)$, the lower
Hessian bound in Proposition~\ref{prop:busemann-comparison}(ii) gives
\begin{align*}
  \frac{d}{dt}|J|^2
  &=2\langle\nabla_{\partial_tF}J,J\rangle
    =-2\nabla^2B_\xi(J,J)\\
  &\leq-2|J|^2.
\end{align*}
Gronwall's inequality, applied for each fixed $s$, gives
\begin{equation}\label{eq:horospherical-contraction}
  |J(s,t)|\leq e^{-t}|\alpha'(s)|.
\end{equation}
For every $t\geq0$, the curve $s\mapsto F(s,t)$ joins
$c_{x,\xi}(t)$ to $c_{y,\xi}(t)$.  Integrating
\eqref{eq:horospherical-contraction} in $s$ therefore gives
\[
  d_X\bigl(c_{x,\xi}(t),c_{y,\xi}(t)\bigr)
  \leq e^{-t}L(\alpha).
\]

To estimate the tangent directions, lift the variation to the curve
\[
  \beta_t(s)=\partial_tF(s,t)\in SX.
\]
Its endpoints are $\varphi^tv$ and $\varphi^tw$.  Under the horizontal--
vertical splitting of $T(SX)$, the components of $\beta_t'(s)$ are
\[
  J(s,t)
  \qquad\text{and}\qquad
  \nabla_J\partial_tF=-\nabla_J\nabla B_\xi.
\]
Because $J$ is tangent to the corresponding horosphere, the
shape-operator bound in Lemma~\ref{lem:horospherical-geometry}(ii),
together with \eqref{eq:horospherical-contraction}, yields
\[
  |\beta_t'(s)|_{\mathrm{Sasaki}}^2
  =|J(s,t)|^2+|\nabla_J\nabla B_\xi|^2
  \leq(1+A^2)e^{-2t}|\alpha'(s)|^2.
\]
Taking the length of $\beta_t$ gives
\[
  d_{SX}\bigl(\varphi^tv,\varphi^tw\bigr)
  \leq\sqrt{1+A^2}\,e^{-t}L(\alpha)
  \longrightarrow0.
\]
Thus $w\in W^{ss}_{SX}(v)$.

\smallskip
\noindent\emph{Step 2: strong stability recovers the endpoint and the
Busemann height.}
Conversely, suppose
$w=\widetilde\Phi(y,\eta)\in W^{ss}_{SX}(v)$.  Since $\pi_X$ is
$1$-Lipschitz by \eqref{eq:footpoint-one-lipschitz},
\[
  d_X\bigl(c_{x,\xi}(t),c_{y,\eta}(t)\bigr)
  \longrightarrow0.
\]
The distance is therefore bounded on $[0,\infty)$: it tends to zero for
large $t$ and is continuous on every compact time interval.  Hence
\eqref{eq:ws-upstairs} gives $\eta=\xi$.  The Busemann difference is
constant along the two rays by \eqref{eq:busemann-height-flow}.  Since
$B_\xi$ is $1$-Lipschitz, for every $t\geq0$,
\begin{align*}
  |B_\xi(x)-B_\xi(y)|
  &=\left|
    B_\xi\bigl(c_{x,\xi}(t)\bigr)
    -B_\xi\bigl(c_{y,\xi}(t)\bigr)
  \right|\\
  &\leq d_X\bigl(c_{x,\xi}(t),c_{y,\xi}(t)\bigr).
\end{align*}
Letting $t\to\infty$ gives $B_\xi(x)=B_\xi(y)$.  Therefore
$(y,\xi)\in\tWss(x,\xi)$, which proves
\eqref{eq:ss-identification-upstairs}.

\medskip
\noindent\emph{Step 3: passage to the quotient.}
Let $q:SX\to SM$ be the covering map and
$z=[x,\xi]$.  Given $z'\in\Wss(z)$, choose
$(y,\xi)\in\tWss(x,\xi)$ with $z'=[y,\xi]$, and set
$w=\widetilde\Phi(y,\xi)$.  By
\eqref{eq:ss-identification-upstairs},
$d_{SX}(\varphi^tv,\varphi^tw)\to0$.  Since $q$ is distance
nonincreasing, it follows that
\[
  d_{SM}\bigl(\varphi^t\overline\Phi(z),
  \varphi^t\overline\Phi(z')\bigr)\longrightarrow0.
\]
Hence
$\overline\Phi(\Wss(z))\subseteq
W^{ss}_{SM}(\overline\Phi(z))$.

Conversely, let
$\bar w\in W^{ss}_{SM}(\overline\Phi(z))$, and choose a lift $w\in SX$
with $q(w)=\bar w$.  The quotient distance is an infimum over deck
transformations.

Because $SM$ is compact and $q$ is a Riemannian covering, there is
$\rho>0$ such that every ball of radius $2\rho$ in $SM$ is geodesically
convex and evenly covered by $q$.  Choose $T>0$ such that
\[
  d_{SM}\bigl(\varphi^tq(v),\varphi^t\bar w\bigr)<\rho
  \qquad\text{for all }t\geq T.
\]
For each $t\geq T$, let $\sigma_t$ be the unique minimizing geodesic in
$SM$ from $\varphi^tq(v)$ to $\varphi^t\bar w$.  Lift $\sigma_t$ from
the initial point $\varphi^tv$, and denote the terminal point of the
lift by $\widehat w(t)$.  Since $q$ is an isometry on the relevant
sheet,
\[
  q(\widehat w(t))=\varphi^t\bar w,
  \qquad
  d_{SX}\bigl(\varphi^tv,\widehat w(t)\bigr)
  =d_{SM}\bigl(\varphi^tq(v),\varphi^t\bar w\bigr).
\]
The unique short geodesic depends continuously on its endpoints, so
$t\mapsto\widehat w(t)$ is continuous.  It is therefore a lift of the
path $t\mapsto\varphi^t\bar w$ on $[T,\infty)$.  At $t=T$, there is a
deck transformation $\gamma\in\Gamma$ such that
\[
  \widehat w(T)=d\gamma(\varphi^Tw).
\]
Both $t\mapsto\widehat w(t)$ and
$t\mapsto d\gamma(\varphi^tw)$ lift the path
$t\mapsto\varphi^t\bar w$ and agree at $T$.  Uniqueness of path lifting
therefore gives
\[
  \widehat w(t)=d\gamma(\varphi^tw)
  \qquad(t\geq T).
\]
Since deck transformations commute with the geodesic flow,
\[
  d_{SX}\bigl(\varphi^tv,
  \varphi^t d\gamma(w)\bigr)\longrightarrow0
\]
and $d\gamma(w)\in W^{ss}_{SX}(v)$.  By
\eqref{eq:ss-identification-upstairs}, there is a point
$(y,\xi)\in\tWss(x,\xi)$ such that
$d\gamma(w)=\widetilde\Phi(y,\xi)$.  Therefore
\[
  \bar w=q(w)=q\bigl(d\gamma(w)\bigr)
  =\overline\Phi([y,\xi])
  \in\overline\Phi\bigl(\Wss(z)\bigr).
\]
Since $\bar w$ was arbitrary, the reverse inclusion follows.
\end{proof}

\subsection{Strong stable foliation boxes and leafwise divergence}

The sets $\Wss(z)$ are the strong stable leaves in the boundary
coordinates on $Z$.  Foliation-box coordinates define leafwise divergence.
The Gauss--Green formula and \eqref{eq:normal-alignment} imply its vanishing
against the limiting measure.

Let $p:X\times\partial_\infty X\to Z$ be the quotient map.  Equip each
lifted weak stable leaf $X\times\{\xi\}$ with the metric $g_X$.  Since
deck transformations are isometries, the restriction of $p$ to each
weak stable leaf is a local isometry.  By
\eqref{eq:tangent-horosphere}, for $z=[x,\xi]$ the tangent space of the
strong stable leaf through $z$ is
\begin{equation}\label{eq:tangent-ss}
  T_z\Wss
  =dp_{(x,\xi)}\bigl(\ker dB_\xi|_x\bigr)
  =dp_{(x,\xi)}
   \bigl((\nabla B_\xi(x))^\perp\bigr).
\end{equation}
Here $dp_{(x,\xi)}$ denotes the differential in the $X$-direction.
The gradient equivariance derived from
\eqref{eq:stable-equivariance} shows that this subspace is independent of
the representative $(x,\xi)$.  Denote the resulting leafwise tangent
bundle by $T\Wss$.

Equation~\eqref{eq:tangent-ss} identifies the tangent spaces but does not
give foliation charts with controlled transverse dependence.

\begin{lemma}
\label{lem:strong-stable-foliation-regularity}
The family $\{\Wss(z):z\in Z\}$ forms the strong stable foliation
$\Wss$ of $Z$.  Its leaves have dimension $m-1$ and are of class
$C^2$.  In strong stable foliation boxes, the leafwise coordinate
changes and their derivatives through order two depend continuously on
the transverse parameter.
\end{lemma}

\begin{proof}
Fix $z_0=[x_0,\xi_0]\in Z$ and set $c_0=B_{\xi_0}(x_0)$.  Choose smooth
coordinates $x=x(y,s)$ near $x_0$, where
$y=(y^1,\ldots,y^{m-1})$, such that the coordinate vector
$\partial_s$ at $x_0$ is $\nabla B_{\xi_0}(x_0)$.  Then
$\partial_sB_{\xi_0}(x_0)=1$.  After shrinking the coordinate neighborhood
and the neighborhood of $\xi_0$,
Proposition~\ref{prop:busemann-comparison}\textup{(iii)} gives
\[
  \partial_sB_\xi(x(y,s))>0.
\]
The level equation can therefore be solved uniquely for $s$.  The implicit
function theorem with parameters gives a connected open set
$P\subset\R^{m-1}$, a neighborhood $T$ of $(\xi_0,c_0)$ in
$\partial_\infty X\times\R$, and a function $\rho$ such that
\[
  B_\xi\bigl(x(y,\rho(y;\xi,c))\bigr)=c,
  \qquad (y,(\xi,c))\in P\times T.
\]
For fixed $(\xi,c)$, the function $y\mapsto\rho(y;\xi,c)$ is $C^2$;
the function $\rho$ and its first two $y$-derivatives are jointly
continuous in $(y,\xi,c)$.  To see the dependence of these derivatives,
write
\[
  b(y,s,\xi):=B_\xi(x(y,s)).
\]
Differentiating $b(y,\rho(y;\xi,c),\xi)=c$ in a plaque direction gives
\[
  \partial_{y^a}\rho
  =-\frac{\partial_{y^a}b}{\partial_sb}.
\]
A second differentiation gives
\[
  \partial^2_{y^a y^b}\rho
  =\frac{
    b_{y^a y^b}+b_{y^a s}\partial_{y^b}\rho
    -b_{y^b s}\partial_{y^a}\rho
    -b_{ss}\partial_{y^a}\rho\,\partial_{y^b}\rho
  }{b_s},
\]
where the derivatives of $b$ are evaluated at
$(y,\rho(y;\xi,c),\xi)$.  The joint continuity established in Section~2
and Proposition~\ref{prop:busemann-comparison}\textup{(iii)} therefore
gives the stated transverse continuity through order two.

Because the action of $\Gamma$ on $X$ is free and properly discontinuous,
the coordinate neighborhood may also be chosen so that its translates are
pairwise disjoint.  The quotient map is then injective on its product with
$\partial_\infty X$.  In this lifted neighborhood, set
\[
  \widetilde U
  =\bigl\{(x(y,s),\xi):y\in P,
  \ (\xi,B_\xi(x(y,s)))\in T\bigr\}.
\]
This set is open.  Writing $\tau=(\xi,c)$, define
\[
  \Theta:P\times T\longrightarrow Z,
  \qquad
  \Theta(y,\tau)
  =\bigl[x\bigl(y,\rho(y;\xi,c)\bigr),\xi\bigr].
\]
In the chosen lift, the inverse map is
\[
  \bigl[x(y,s),\xi\bigr]
  \longmapsto
  \bigl(y,(\xi,B_\xi(x(y,s)))\bigr).
\]
The uniqueness of $\rho$ shows that $\Theta(P\times T)$ is the
image of $\widetilde U$ in $Z$.  Since the quotient map is open,
$U:=\Theta(P\times T)$ is open; the displayed inverse is continuous, so
$\Theta:P\times T\to U$ is a homeomorphism.  For each
$\tau=(\xi,c)$, the image of
$P\times\{\tau\}$ is a connected open subset of the strong stable set
with endpoint $\xi$ and height $c$.  This set is the \emph{plaque} with
transverse parameter $\tau$.

Such boxes cover $Z$.  On a connected component of the overlap of two
boxes, the coordinate change has the form
\[
  (y,\tau)\longmapsto\bigl(F(y,\tau),G(\tau)\bigr).
\]
The strong stable set through a point is intrinsic, so a coordinate
change sends each connected plaque component into a plaque and its new
transverse label is independent of $y$.  Thus the transverse coordinate is
$G(\tau)$, with $G$ continuous.  The graph construction above shows that
$F$ is $C^2$ in $y$ and that its first two $y$-derivatives are continuous
in $(y,\tau)$.  These charts define a foliation with the stated leafwise
and transverse regularity.
\end{proof}

In a foliation box $P\times T$, write the induced plaque metric as
\[
  g_\tau
  =\sum_{a,b=1}^{m-1}g_{ab}(y,\tau)\,dy^a\,dy^b.
\]
The coefficients $g_{ab}$ are $C^1$ in $y$, and both $g_{ab}$ and
$\partial_{y^c}g_{ab}$ depend continuously on $(y,\tau)$.  Thus the
plaque-volume form is
\begin{equation}\label{eq:plaque-volume}
  dV_\tau(y)=Q(y,\tau)\,dy,
  \qquad
  Q(y,\tau)=\sqrt{\det(g_{ab}(y,\tau))}>0.
\end{equation}
The function $Q$ is positive and $C^1$ along plaques, and both $Q$ and
its first $y$-derivatives are continuous on $P\times T$.

\begin{definition}
\label{def:strong-stable-divergence}
Let $Y$ be a continuous section of $T\Wss$ that is $C^1$ along every
strong stable leaf.  For $z\in Z$, let $L_z=\Wss(z)$ and define
\[
  (\nabla^{ss}Y)_z(V)
  :=\nabla^{L_z}_V(Y|_{L_z}),
  \qquad V\in T_zL_z,
\]
where $\nabla^{L_z}$ is the Levi--Civita connection of the induced metric
on $L_z$.  Assume that $\nabla^{ss}Y$ is continuous on $Z$.  The
\emph{strong stable divergence} of $Y$ is
\[
  \divss Y(z):=\operatorname{tr}_{T_zL_z}(\nabla^{ss}Y)_z.
\]
Equivalently, in a strong stable foliation box, if
$Y=\sum_{a=1}^{m-1}Y^a\partial_{y^a}$, then
\[
  \divss Y=\frac1Q\sum_{a=1}^{m-1}\partial_{y^a}(QY^a),
\]
where $Q$ is the plaque-volume density in \eqref{eq:plaque-volume}.
\end{definition}

The trace is independent of the foliation box, and the coordinate formula
shows that $\divss Y$ is continuous on $Z$.

The pullback of $Y$ to a lifted weak stable leaf is tangent to the
horospheres but need not be $C^1$ in the transverse direction.

For $\xi\in\partial_\infty X$, write
\[
  \mathcal H_\xi:=\{H_{\xi,t}:t\in\R\},
  \qquad
  H_{\xi,t}:=B_\xi^{-1}(t),
\]
and let $H_\xi(x):=H_{\xi,B_\xi(x)}$ be the horosphere through $x$.
Suppose that $Y$ is a continuous vector field with
$Y(x)\in T_xH_\xi(x)$ for every $x\in X$ and is $C^1$ along every
horosphere.  For $V\in T_xH_\xi(x)$, define
\[
  (\nabla^{\mathcal H_\xi}Y)_x(V)
  :=\nabla^{H_\xi(x)}_V(Y|_{H_\xi(x)}),
\]
where $\nabla^{H_\xi(x)}$ is the Levi--Civita connection of the induced
metric on $H_\xi(x)$.  Its trace is the \emph{horospherical divergence}
\[
  \diver_{\mathcal H_\xi}Y(x)
  :=\operatorname{tr}_{T_xH_\xi(x)}
    (\nabla^{\mathcal H_\xi}Y)_x.
\]
\begin{lemma}
\label{lem:horosphere-div}
Fix $\xi\in\partial_\infty X$.  Let $Y$ be a continuous vector field
tangent to $\mathcal H_\xi$ and $C^1$ along every horosphere.  If
$\nabla^{\mathcal H_\xi}Y$ is continuous on $X$, then
$\diver_{\mathcal H_\xi}Y$ is continuous and every smooth domain
$\Omega\Subset X$ with outward unit normal $\nu$ satisfies
\begin{equation}\label{eq:horospherical-gauss-green}
  \int_\Omega \diver_{\mathcal H_\xi}Y\,dV
  =\int_{\partial\Omega}\langle Y,\nu\rangle\,dA.
\end{equation}
\end{lemma}

\begin{proof}
The continuity of $\diver_{\mathcal H_\xi}Y$ follows by taking the trace
of $\nabla^{\mathcal H_\xi}Y$.

Apply Lemma~\ref{lem:horosphere-tangent-approximation} with
$B=B_\xi$.  For a fixed smooth domain $\Omega\Subset X$, it gives
$C^1$ vector fields $Y_k$ such that, uniformly on $\overline\Omega$,
\[
  Y_k\longrightarrow Y,
  \qquad
  \diver_XY_k\longrightarrow\diver_{\mathcal H_\xi}Y.
\]
The classical divergence theorem applied to $Y_k$, followed by passage
to the limit, gives
\begin{align*}
  \int_\Omega\diver_{\mathcal H_\xi}Y\,dV
  &=\lim_{k\to\infty}\int_\Omega\diver_XY_k\,dV\\
  &=\lim_{k\to\infty}\int_{\partial\Omega}
    \langle Y_k,\nu\rangle\,dA\\
  &=\int_{\partial\Omega}\langle Y,\nu\rangle\,dA.
\end{align*}
\end{proof}

For the lift of a strong stable vector field,
Lemma~\ref{lem:horosphere-div} gives the Gauss--Green identity, while
Proposition~\ref{prop:yau-defects}(ii) shows that the normalized boundary
flux tends to zero.

\begin{proposition}
\label{prop:ss-div}
Let $(M^m,g)$, $(X,g_X)$, $\Gamma$, and $Z$ be as in
Section~\ref{sec:suspension-localization}, and suppose that
$\hiso(X)=m-1$.  Fix $\xi_0\in\partial_\infty X$, and let
$\{\Omega_j\}$ satisfy \eqref{eq:isoperimetric-sequence}.  Define
$\mu_j$ by \eqref{eq:mu-j}.  Then every weak limit $\mu$ of
$\{\mu_j\}$ satisfies
\[
  \int_Z\divss Y\,d\mu=0
\]
for every vector field $Y$ satisfying the hypotheses of
Definition~\ref{def:strong-stable-divergence}.
\end{proposition}

\begin{proof}
Pass to a subsequence, still indexed by $j$, such that
$\mu_j\rightharpoonup\mu$, and fix $Y$ as in
Definition~\ref{def:strong-stable-divergence}.  The map
$\iota_{\xi_0}:X\to Z$, $\iota_{\xi_0}(x)=[x,\xi_0]$, is a
leafwise local isometry.  Define $Y_0$ by
\[
  (d\iota_{\xi_0})_xY_0(x)=Y(\iota_{\xi_0}(x)).
\]
Then
\[
  |Y_0(x)|=|Y(\iota_{\xi_0}(x))|\leq\|Y\|_\infty.
\]
Equation~\eqref{eq:tangent-ss} gives
\begin{equation}\label{eq:Y0-tangent-horospheres}
  \bigl\langle Y_0,\nabla B_{\xi_0}\bigr\rangle=0.
\end{equation}
Since $\iota_{\xi_0}$ restricts to a local isometry on each horosphere,
$Y_0$ is $C^1$ along the horospheres and
$\nabla^{\mathcal H_{\xi_0}}Y_0$ is continuous.  Thus
Lemma~\ref{lem:horosphere-div} applies, and
Definition~\ref{def:strong-stable-divergence} gives
\[
  (\divss Y)(\iota_{\xi_0}(x))
  =(\diver_{\mathcal H_{\xi_0}}Y_0)(x).
\]

Set $v_j=\vol(\Omega_j)$, and let $\nu_j$ be the outward unit normal to
$\partial\Omega_j$.  Equations~\eqref{eq:mu-j} and
\eqref{eq:horospherical-gauss-green} give
\begin{align*}
  \int_Z\divss Y\,d\mu_j
  &=\frac1{v_j}\int_{\Omega_j}
  \diver_{\mathcal H_{\xi_0}}Y_0\,dV\\
  &=\frac1{v_j}\int_{\partial\Omega_j}
  \langle Y_0,\nu_j\rangle\,dA.
\end{align*}
By \eqref{eq:Y0-tangent-horospheres},
\[
  \langle Y_0,\nu_j\rangle
  =\bigl\langle Y_0,\nu_j-\nabla B_{\xi_0}\bigr\rangle.
\]
Therefore the Cauchy--Schwarz inequality gives
\begin{align*}
  \left|\int_Z\divss Y\,d\mu_j\right|
  &\leq \frac{\|Y\|_\infty}{v_j}
  \area(\partial\Omega_j)^{1/2}
  \left(
  \int_{\partial\Omega_j}
  |\nu_j-\nabla B_{\xi_0}|^2\,dA
  \right)^{1/2}\\
  &=\|Y\|_\infty
  \left(
  \frac{\area(\partial\Omega_j)}{\vol(\Omega_j)}
  \right)^{1/2}
  \left(
  \frac1{\vol(\Omega_j)}
  \int_{\partial\Omega_j}
  |\nu_j-\nabla B_{\xi_0}|^2\,dA
  \right)^{1/2}.
\end{align*}
The first factor is bounded by \eqref{eq:isoperimetric-sequence}, and the
second tends to zero by Proposition~\ref{prop:yau-defects}(ii).  Hence
\[
  \int_Z\divss Y\,d\mu_j\longrightarrow0.
\]
Since $\divss Y$ is continuous on the compact space $Z$, weak convergence
now gives
\[
  \int_Z\divss Y\,d\mu
  =\lim_{j\to\infty}\int_Z\divss Y\,d\mu_j
  =0.
\]
\end{proof}

\begin{corollary}
\label{cor:limiting-measure}
Under the hypotheses of Proposition~\ref{prop:ss-div}, the sequence
$\{\mu_j\}$ has weakly convergent subsequences.  Every weak limit $\mu$
is a Borel probability measure on $Z$ such that:
\begin{enumerate}[label=\textup{(\roman*)}]
\item For the function $D$ defined in \eqref{eq:D-on-Z},
\begin{equation}\label{eq:support-zero-set}
  \supp\mu\subseteq D^{-1}(0).
\end{equation}

\item For every vector field $Y$ satisfying the hypotheses of
Definition~\ref{def:strong-stable-divergence}, one has
\begin{equation}\label{eq:ss-divergence}
  \int_Z \divss Y\,d\mu=0.
\end{equation}
\end{enumerate}
\end{corollary}

\begin{proof}
Weak sequential compactness was established in
Section~\ref{sec:suspension-localization}.  Proposition~\ref{prop:limit-support}
gives \textup{(i)}, and Proposition~\ref{prop:ss-div} gives
\textup{(ii)}.
\end{proof}

\section{Plaque-volume disintegration and full support}
\label{sec:disintegration-full-support}
Let $\mu$ be a weak limit given by
Corollary~\ref{cor:limiting-measure}.  The identity
\eqref{eq:ss-divergence} implies that the conditional measures along plaques
be multiples of their Riemannian volume measures.  Hence $\supp\mu$ is
saturated by strong stable leaves.  Since every such leaf is dense,
$\supp\mu=Z$.

\subsection{Plaque-volume disintegration}

\begin{lemma}
\label{lem:constant-measure}
Let $P\subset\R^d$ be a connected open set, and let $\sigma$ be a locally
finite nonnegative Borel measure on $P$.  Suppose that
\begin{equation}\label{eq:zero-measure-derivative}
  \int_P\partial_a\psi\,d\sigma=0
\end{equation}
for every $\psi\in C_c^\infty(P)$ and every $a=1,\ldots,d$.  Then there
is a constant $c\geq0$ such that $\sigma=c\,dy$ on $P$.
\end{lemma}

\begin{proof}
Fix concentric open balls $P'\Subset P''\Subset P$.  For all sufficiently
small $\varepsilon>0$, the convolution
$f_\varepsilon=\sigma*\eta_\varepsilon$ is defined and smooth on $P''$.
Equation~\eqref{eq:zero-measure-derivative} says that every distributional
derivative of $\sigma$ vanishes.  Hence
\[
  \partial_af_\varepsilon
  =\sigma*(\partial_a\eta_\varepsilon)=0
  \qquad (a=1,\ldots,d)
\]
on $P''$, and $f_\varepsilon$ is constant there; denote this constant by
$c_\varepsilon$.

Choose $\theta\in C_c^\infty(P')$ with $\int_{P'}\theta\,dy=1$.  The
weak convergence $f_\varepsilon\,dy\rightharpoonup\sigma$ on $P'$ gives
\[
  c_\varepsilon
  =\int_{P'}\theta f_\varepsilon\,dy
  \longrightarrow\int_{P'}\theta\,d\sigma=:c_{P'}.
\]
For every $\psi\in C_c^\infty(P')$, we therefore have
\[
  \int_{P'}\psi\,d\sigma
  =\lim_{\varepsilon\to0}\int_{P'}\psi f_\varepsilon\,dy
  =c_{P'}\int_{P'}\psi\,dy.
\]
Thus $\sigma=c_{P'}\,dy$ on $P'$.  The constants agree on overlapping
balls.  Since every connected open subset of $\R^d$ is path connected, a
chain of overlapping balls joins any two such balls in $P$.  Therefore
$\sigma=c\,dy$ on $P$ for one constant $c\geq0$.
\end{proof}

\begin{proposition}
\label{prop:local-disintegration}
Let $\mu$ be a finite Radon measure on $Z$ such that
\begin{equation}\label{eq:section-five-divergence-hypothesis}
  \int_Z\divss Y\,d\mu=0
\end{equation}
for every $Y$ satisfying Definition~\ref{def:strong-stable-divergence}.
Let $U=P\times T$ be a strong stable foliation box.  Choose a
connected open set $P_0\Subset P$ and an open set $T_0\Subset T$, and
write $U_0=P_0\times T_0$.  Then there is a finite Radon measure $\nu$
on $T_0$ such that
\begin{equation}\label{eq:local-disintegration}
  \int_{U_0}f\,d\mu
  =\int_{T_0}\left(\int_{P_0}f(y,\tau)\,dV_\tau(y)\right)d\nu(\tau)
\end{equation}
for every $f\in C_c(U_0)$.
\end{proposition}

\begin{proof}
If $\mu(U_0)=0$, take $\nu=0$.  Assume $\mu(U_0)>0$.

Set
\[
  M:=\mu(U_0),
  \qquad
  \overline\mu:=M^{-1}(\mu|_{U_0}).
\]
Then $\overline\mu$ is a Radon probability measure on
$U_0=P_0\times T_0$.  Let $\pi_T:U_0\to T_0$ be the transverse
projection and set
\[
  \lambda:=(\pi_T)_*\overline\mu.
\]

The spaces $P_0$ and $T_0$ are locally compact and second countable, and
\[
  \mathcal B(P_0\times T_0)
  =\mathcal B(P_0)\otimes\mathcal B(T_0).
\]
The $\sigma$-algebra $\mathcal B(P_0)$ is countably generated, and the
pushforward of
$\overline\mu$ under the projection onto $P_0$ is Radon and hence has a
compact approximating class.
By the product disintegration theorem
(see \cite{Bogachev2007}*{Theorem~10.4.14}), there is a
$\lambda$-measurable family of nonnegative Borel measures
$\{\mu_\tau\}$ on $P_0$ such that, for every Borel set
$E\subset P_0\times T_0$,
\begin{equation}\label{eq:conditional-disintegration}
  \overline\mu(E)
  =\int_{T_0}\mu_\tau(E_\tau)\,d\lambda(\tau),
  \qquad
  E_\tau:=\{y\in P_0:(y,\tau)\in E\}.
\end{equation}
In particular,
$\tau\mapsto\mu_\tau(A)$ is $\lambda$-measurable for every Borel set
$A\subset P_0$.

Let $B\subset T_0$ be Borel.  Applying
\eqref{eq:conditional-disintegration} to $E=P_0\times B$ and using
$\lambda=(\pi_T)_*\overline\mu$ gives
\[
  \lambda(B)
  =\overline\mu(P_0\times B)
  =\int_B\mu_\tau(P_0)\,d\lambda(\tau).
\]
Therefore
\[
  \int_B\bigl(\mu_\tau(P_0)-1\bigr)\,d\lambda(\tau)=0
\]
for every Borel set $B\subset T_0$, and hence
\begin{equation}\label{eq:conditional-probability}
  \mu_\tau(P_0)=1
  \qquad\text{for $\lambda$-almost every $\tau$.}
\end{equation}
$\mu_\tau$ is therefore a probability measure on $P_0$ for
$\lambda$-almost every $\tau$.  Choose $y_*\in P_0$ and replace
$\mu_\tau$ by $\delta_{y_*}$ on the exceptional $\lambda$-null set.
This changes neither the measurability of the family nor
\eqref{eq:conditional-disintegration}.  After this modification, every
$\mu_\tau$ is a probability measure on $P_0$.

For the original measure, set
\[
  \widehat\mu:=(\pi_T)_*(\mu|_{U_0})=M\lambda.
\]
In particular, $\lambda$ and $\widehat\mu$ have the same null sets.
Multiplying the function form of
\eqref{eq:conditional-disintegration} by $M$ gives
\begin{equation}\label{eq:mu-disintegration}
  \int_{U_0}f\,d\mu
  =\int_{T_0}\left(\int_{P_0}f(y,\tau)\,d\mu_\tau(y)\right)
  d\widehat\mu(\tau)
\end{equation}
for every $f\in C_c(U_0)$.

After division by the plaque-volume density, each conditional measure has
zero distributional gradient.  Write
$dV_\tau=Q(y,\tau)\,dy$ as in \eqref{eq:plaque-volume}.  Fix
$\psi\in C_c^\infty(P_0)$, $1\leq a\leq m-1$, and
$\chi\in C_c(T_0)$, and define
\[
  Y_{\psi,a,\chi}(y,\tau)
  :=\chi(\tau)Q(y,\tau)^{-1}\psi(y)
  \frac{\partial}{\partial y^a}.
\]
Its support is contained in
$\supp\psi\times\supp\chi\Subset U_0$.  Since $Q^{-1}$ is $C^1$ along
the plaques and its first plaque derivatives depend continuously on
$(y,\tau)$, extension by zero gives a global field satisfying
Definition~\ref{def:strong-stable-divergence}.  Since
$QY_{\psi,a,\chi}^a=\chi\psi$, the coordinate divergence formula gives
\[
  \divss Y_{\psi,a,\chi}
  =Q(y,\tau)^{-1}\partial_a
    \bigl(QY_{\psi,a,\chi}^a\bigr)
  =\chi(\tau)Q(y,\tau)^{-1}\partial_a\psi(y).
\]
Set
\[
  h_{\psi,a}(\tau)
  :=\int_{P_0}Q(y,\tau)^{-1}\partial_a\psi(y)\,d\mu_\tau(y).
\]
The integrand is jointly Borel and uniformly bounded on
$\supp\psi\times\overline{T_0}$.  The measurability of the conditional
kernel implies that $h_{\psi,a}$ is measurable and
bounded.  Equations
\eqref{eq:section-five-divergence-hypothesis} and
\eqref{eq:mu-disintegration} give
\[
  0=\int_Z\divss Y_{\psi,a,\chi}\,d\mu
  =\int_{T_0}\chi(\tau)h_{\psi,a}(\tau)\,d\widehat\mu(\tau)
  \qquad\text{for every }\chi\in C_c(T_0).
\]
The signed Radon measure $h_{\psi,a}\,\widehat\mu$ vanishes.  Hence
\begin{equation}\label{eq:conditional-zero-derivative}
  \int_{P_0}Q(y,\tau)^{-1}\partial_a\psi(y)\,d\mu_\tau(y)=0
\end{equation}
for $\widehat\mu$-almost every $\tau$.  The exceptional set may still
depend on $\psi$ and $a$.

To make the exceptional set independent of the test function, choose a
compact exhaustion
$K_n\Subset\operatorname{int}K_{n+1}$ of $P_0$.  For each $n$, the space
\[
  \mathcal E_n
  :=\{\psi\in C_c^\infty(P_0):\supp\psi\subset K_n\}
\]
is separable in the $C^1$ norm: it embeds isometrically into the separable
space $C(K_n)^m$ by
\[
  \psi\longmapsto
  (\psi,\partial_1\psi,\ldots,\partial_{m-1}\psi).
\]
Choose a countable $C^1$-dense subset
$\mathcal D_n\subset\mathcal E_n$.  Intersect the full-measure sets on
which \eqref{eq:conditional-zero-derivative} holds for
$\psi\in\bigcup_n\mathcal D_n$ and $a=1,\ldots,m-1$.  Because the
intersection is countable, the resulting set $T_0'\subset T_0$ has full
$\widehat\mu$-measure.

Fix $\tau\in T_0'$ and $\psi\in C_c^\infty(P_0)$.  Choose $n$ so that
$\supp\psi\subset\operatorname{int}K_n$, and choose
$\psi_j\in\mathcal D_n$ with
$\psi_j\to\psi$ in $C^1$.  Since $\mu_\tau(P_0)=1$ and
$Q(\cdot,\tau)^{-1}$ is bounded on $K_n$,
\[
  \left|\int_{P_0}Q(y,\tau)^{-1}
  \partial_a(\psi_j-\psi)(y)\,d\mu_\tau(y)\right|
  \leq
  \sup_{K_n}Q(\cdot,\tau)^{-1}
  \|\partial_a(\psi_j-\psi)\|_\infty
  \longrightarrow0.
\]
Passing to the limit shows that \eqref{eq:conditional-zero-derivative}
holds simultaneously
for every $\tau\in T_0'$, every $\psi\in C_c^\infty(P_0)$, and every
$a=1,\ldots,m-1$.

Fix $\tau\in T_0'$ and define
\[
  d\sigma_\tau:=Q(\cdot,\tau)^{-1}\,d\mu_\tau.
\]
This is a finite nonnegative Borel measure on $P_0$, because
$Q(\cdot,\tau)^{-1}$ is bounded on
$\overline{P_0}$ and $\mu_\tau(P_0)=1$.
Equation~\eqref{eq:conditional-zero-derivative} says that every
distributional derivative of $\sigma_\tau$ vanishes.  Since $P_0$ is
connected, Lemma~\ref{lem:constant-measure} gives a constant
$c(\tau)\geq0$ such that
\[
  d\sigma_\tau=c(\tau)\,dy,
  \qquad
  d\mu_\tau=c(\tau)\,dV_\tau.
\]
Let
\[
  v(\tau):=\vol_\tau(P_0)=\int_{P_0}Q(y,\tau)\,dy.
\]
Because $\mu_\tau$ is a probability measure,
\[
  1=\mu_\tau(P_0)=c(\tau)v(\tau),
  \qquad
  c(\tau)=v(\tau)^{-1}.
\]
For $\widehat\mu$-almost every $\tau$,
\begin{equation}\label{eq:normalized-plaque-volume}
  d\mu_\tau=v(\tau)^{-1}\,dV_\tau.
\end{equation}

The function $v$ is continuous and bounded away from zero.  Continuity
and positivity of $Q$ on the compact set
$\overline{P_0}\times\overline{T_0}$ give constants
$0<q_-\leq q_+<\infty$ such that
\[
  q_-\leq Q(y,\tau)\leq q_+
  \qquad
  ((y,\tau)\in\overline{P_0}\times\overline{T_0}).
\]
Hence
\[
  q_-|P_0|\leq v(\tau)\leq q_+|P_0|,
\]
and the uniform continuity of $Q$ on the same compact set implies the
continuity of $v$.  Define
\[
  d\nu(\tau):=v(\tau)^{-1}\,d\widehat\mu(\tau).
\]
The function $v^{-1}$ is bounded and continuous, so $\nu$ is a finite
Radon measure.  For $f\in C_c(U_0)$,
equations~\eqref{eq:mu-disintegration} and
\eqref{eq:normalized-plaque-volume} give
\[
  \int_{U_0}f\,d\mu
  =\int_{T_0}v(\tau)^{-1}
    \left(\int_{P_0}f(y,\tau)\,dV_\tau(y)\right)d\widehat\mu(\tau)
  =\int_{T_0}\left(\int_{P_0}f(y,\tau)\,dV_\tau(y)\right)d\nu(\tau).
\]
This is \eqref{eq:local-disintegration}.
\end{proof}

Equation~\eqref{eq:local-disintegration} is an equality of Radon measures
on $U_0$.  Consequently, if $A\subset U_0$ is Borel and
$A_\tau:=\{y\in P_0:(y,\tau)\in A\}$, then
\begin{equation}\label{eq:borel-disintegration}
  \mu(A)=\int_{T_0}\vol_\tau(A_\tau)\,d\nu(\tau).
\end{equation}
In particular, for Borel sets $E\subset P_0$ and $V\subset T_0$,
\begin{equation}\label{eq:rectangle-disintegration}
  \mu(E\times V)=\int_V\vol_\tau(E)\,d\nu(\tau).
\end{equation}
Since $dV_\tau=Q(\cdot,\tau)\,dy$ with $Q$ positive and continuous, every
nonempty open set $E\Subset P_0$ and every $V\Subset T_0$ satisfy
\[
  \inf_{\tau\in V}\vol_\tau(E)>0.
\]

\subsection{Propagation of support along leaves}

The local product description of the support extends along each connected
strong stable leaf through overlapping foliation boxes.

\begin{proposition}
\label{prop:support-saturation}
Let $\mu$ be a finite Radon measure on $Z$ satisfying
\eqref{eq:section-five-divergence-hypothesis}.  For every relatively
compact subbox
$U_0=P_0\times T_0$ considered in
Proposition~\ref{prop:local-disintegration}, and for the corresponding
transverse measure $\nu$,
\begin{equation}\label{eq:local-support-product}
  \supp\mu\cap U_0=P_0\times\supp_{T_0}\nu,
\end{equation}
where $\supp_{T_0}\nu$ denotes the support of $\nu$ in the relative
topology of $T_0$.  Moreover, the support of $\mu$ is saturated by
strong stable leaves: for every $z\in\supp\mu$,
\[
  \Wss(z)\subseteq\supp\mu.
\]
\end{proposition}

\begin{proof}
Set $S:=\supp\mu$.  Since $U_0$ is
open in $Z$, the relative support of $\mu|_{U_0}$ is $S\cap U_0$.
If $\tau_0\notin\supp_{T_0}\nu$, there is an open neighborhood
$V\subset T_0$ of $\tau_0$ with $\nu(V)=0$.  By
\eqref{eq:rectangle-disintegration},
\[
  \mu(P_0\times V)=\int_V\vol_\tau(P_0)\,d\nu(\tau)=0.
\]
The zero-measure open set $P_0\times V$ contains the plaque
$P_0\times\{\tau_0\}$, so
\[
  S\cap U_0\subseteq P_0\times\supp_{T_0}\nu.
\]

For the reverse inclusion, fix
$(y_0,\tau_0)\in P_0\times\supp_{T_0}\nu$, and let $O$ be an arbitrary
open neighborhood of this point in $U_0$.  To prove that
$(y_0,\tau_0)\in S$, it is enough to prove $\mu(O)>0$.  Product
rectangles form a basis in the box coordinates, so there are open sets
$P_1\Subset P_0$ and $V\Subset T_0$ such that
\[
  (y_0,\tau_0)\in P_1\times V
  \quad\text{and}\quad
  \overline{P_1}\times\overline V\subset O.
\]
Positivity and continuity of $Q$ give
\[
  c:=|P_1|\min_{\overline{P_1}\times\overline V}Q>0,
  \qquad \vol_\tau(P_1)\geq c\quad(\tau\in V).
\]
Since $\tau_0\in\supp_{T_0}\nu$, we also have $\nu(V)>0$.  Hence
\[
  \mu(O)
  \geq\mu(P_1\times V)
  =\int_V\vol_\tau(P_1)\,d\nu(\tau)
  \geq c\nu(V)>0.
\]
Since $O$ was arbitrary, $(y_0,\tau_0)\in S$, which proves
\eqref{eq:local-support-product}.

Now fix $z\in S$ and put $L:=\Wss(z)$.  The set $S\cap L$ is nonempty
and closed in $L$.  For $w\in S\cap L$, choose a relatively compact
subbox $U_0=P_0\times T_0$ containing $w=(y_0,\tau_0)$.  The local
identity gives
\[
  w\in S\cap U_0
  \quad\Longrightarrow\quad
  \tau_0\in\supp_{T_0}\nu
  \quad\Longrightarrow\quad
  P_0\times\{\tau_0\}\subset S.
\]
The plaque $P_0\times\{\tau_0\}$ is an open neighborhood of $w$ in
$L$.  Thus $S\cap L$ is open as well as closed in the connected leaf
$L$, and therefore $S\cap L=L$.
\end{proof}

\begin{lemma}
\label{lem:all-leaves-dense}
Let $(M^m,g)$, $m\geq2$, be a closed connected Riemannian manifold with
$-A^2\leq\secg_g\leq-1$ for some $A\geq1$.  Let $(X,g_X)$ be its
universal Riemannian cover with deck group $\Gamma$, and set
$Z=(X\times\partial_\infty X)/\Gamma$.  Then every strong stable leaf is
dense in $Z$.
\end{lemma}

\begin{proof}
The bound $\secg_{g_X}\leq-1$ implies that $X$ satisfies the visibility
axiom, so $M=X/\Gamma$ is a visibility manifold
(see \cite{EO73}*{Definition~4.1 and the following discussion}).
Eberlein's minimality theorem states that a visibility manifold is compact
if and only if every strong stable set is dense in its unit tangent bundle
(see \cite{Eberlein1973}*{Theorem~6.1}).  Hence, for every $z\in Z$,
\[
  \overline{W^{ss}_{SM}(\overline\Phi(z))}=SM.
\]
By Proposition~\ref{prop:stable-leaf-identification},
$\overline\Phi(\Wss(z))=W^{ss}_{SM}(\overline\Phi(z))$.  Since
$\overline\Phi:Z\to SM$ is a homeomorphism, it follows that
$\overline{\Wss(z)}=Z$.
\end{proof}

Proposition~\ref{prop:support-saturation} and
Lemma~\ref{lem:all-leaves-dense} imply full support for the limiting
measures of Section~\ref{sec:suspension-localization}.

\begin{proposition}
\label{prop:full-support}
Under the hypotheses of Proposition~\ref{prop:ss-div}, every weak limit
$\mu$ of $\{\mu_j\}$ satisfies
\[
  \supp\mu=Z.
\]
\end{proposition}

\begin{proof}
Corollary~\ref{cor:limiting-measure} shows that $\mu$ is a probability
measure and satisfies \eqref{eq:section-five-divergence-hypothesis}.
Thus $S:=\supp\mu$ is nonempty, and
Proposition~\ref{prop:support-saturation} shows that $S$ contains the
strong stable leaf through each of its points.  Choose $z\in S$.
Lemma~\ref{lem:all-leaves-dense} gives
$Z=\overline{\Wss(z)}\subset S$, because $S$ is closed.  Hence $S=Z$.
\end{proof}

For the measure in Corollary~\ref{cor:limiting-measure},
Proposition~\ref{prop:full-support} and
\eqref{eq:support-zero-set} give $D^{-1}(0)=Z$.  This is the equality in
Busemann Laplacian comparison used in the proof of the main theorem.

\section{Proof of the main theorem}
\label{sec:proof-main-theorem}

\begin{lemma}
\label{lem:busemann-rigidity}
Let $(X^m,g_X)$ be a complete simply connected Riemannian manifold,
$m\geq2$, with
\[
  -A^2\leq\secg_{g_X}\leq-1
\]
for some $A\geq1$.  Suppose that a Busemann function $B\in C^2(X)$
satisfies
\[
  |\nabla B|=1,
  \qquad
  \Delta B=m-1.
\]
Then $(X,g_X)\cong\Hh^m(-1)$.
\end{lemma}

\begin{proof}
By Proposition~\ref{prop:busemann-comparison}, the symmetric tensor
\[
  \mathcal E
  :=\nabla^2B-\bigl(g_X-dB\otimes dB\bigr)
\]
is positive semidefinite.  Its trace is
\[
  \operatorname{tr}_{g_X}\mathcal E
  =\Delta B-\bigl(m-|\nabla B|^2\bigr)=0.
\]
Hence $\mathcal E=0$, and
\begin{equation}\label{eq:busemann-hessian-equality}
  \nabla^2B=g_X-dB\otimes dB.
\end{equation}

In smooth local coordinates, this identity is
\[
  \partial_i\partial_jB
  =\Gamma_{ij}^k\partial_kB+(g_X)_{ij}
   -(\partial_iB)(\partial_jB).
\]
Starting with $B\in C^2$, induction gives $B\in C^k$ for every $k$.
Thus $B$ is smooth.  If $T=\nabla B$, then
\eqref{eq:busemann-hessian-equality} is equivalent to
\begin{equation}\label{eq:gradient-busemann-covariant-derivative}
  \nabla_VT=V-\langle V,T\rangle T
  \qquad (V\in TX).
\end{equation}
In particular, $\nabla_TT=0$, and the integral curves of $T$ are
unit-speed geodesics.

Set $H=B^{-1}(0)$, and let $\Psi_t$ be the flow of $T$.  Completeness of
$X$ implies that these geodesics, and hence $\Psi_t$, are defined for every
$t\in\R$.  Along an integral curve,
\[
  \frac{d}{dt}B(\Psi_t(x))
  =\langle\nabla B,T\rangle=1,
\]
and therefore
\begin{equation}\label{eq:busemann-along-gradient-flow}
  B(\Psi_t(x))=B(x)+t.
\end{equation}
Consequently,
\[
  \Psi:\R\times H\longrightarrow X,
  \qquad
  \Psi(t,y)=\Psi_t(y),
\]
is a diffeomorphism with inverse
\[
  x\longmapsto
  \bigl(B(x),\Psi_{-B(x)}(x)\bigr).
\]
In particular, $H$ is connected.

Let $g_t$ be the pullback to $H$ of the metric induced on $B^{-1}(t)$.
Extend $Y_1,Y_2\in TH$ by the flow, so that $[T,Y_i]=0$.  The
torsion-free identity $\nabla_TY_i=\nabla_{Y_i}T$ and
\eqref{eq:gradient-busemann-covariant-derivative} give
\begin{align*}
  \frac{d}{dt}g_t(Y_1,Y_2)
  &=\langle\nabla_TY_1,Y_2\rangle
    +\langle Y_1,\nabla_TY_2\rangle\\
  &=\langle\nabla_{Y_1}T,Y_2\rangle
    +\langle Y_1,\nabla_{Y_2}T\rangle\\
  &=2g_t(Y_1,Y_2).
\end{align*}
Thus $g_t=e^{2t}g_H$, where $g_H=g_0$.  Since $T$ is unit and
orthogonal to the level sets of $B$,
\begin{equation}\label{eq:exponential-warped-product}
  \Psi^*g_X=dt^2+e^{2t}g_H.
\end{equation}

The hypersurface $H$ is closed in $X$.  If a sequence is Cauchy for the
intrinsic distance of $H$, it is Cauchy in $X$ and therefore converges to a
point of $H$; local equivalence of the intrinsic and ambient distances then
gives convergence in $H$.  Thus $(H,g_H)$ is complete.  The space $H$ is
also simply connected because $\R\times H\cong X$ and $X$ is simply
connected.

When $m=2$, the one-dimensional manifold $H$ is flat.  Suppose $m\geq3$.
For a two-plane $\sigma\subset T_yH$, let $\sigma_t$ be its image in the
tangent space of $\{t\}\times H$.  The warped-product curvature formula
(see \cite{BishopONeill1969}*{Section~7, formula preceding
Corollary~7.10}) gives
\begin{equation}\label{eq:tangential-warped-curvature}
  \secg_{g_X}(\sigma_t)
  =e^{-2t}\secg_{g_H}(\sigma)-1.
\end{equation}
The upper curvature bound implies $\secg_{g_H}(\sigma)\leq0$.  If
$\secg_{g_H}(\sigma)<0$, then the right-hand side of
\eqref{eq:tangential-warped-curvature} tends to $-\infty$ as
$t\to-\infty$, contrary to $\secg_{g_X}\geq-A^2$.  Hence $g_H$ is flat.

A complete simply connected flat manifold is Euclidean, so
$(H,g_H)\cong(\R^{m-1},g_{\mathrm{Euc}})$.  Equation
\eqref{eq:exponential-warped-product}, under the identification by $\Psi$,
becomes
\[
  g_X=dt^2+e^{2t}g_{\mathrm{Euc}}.
\]
After the change of variable $r=e^{-t}$,
\[
  g_X=\frac{dr^2+g_{\mathrm{Euc}}}{r^2},
\]
the upper-half-space metric on $\Hh^m(-1)$.
\end{proof}

\begin{proof}[Proof of Theorem~\ref{thm:main}]
Let $(X,g_X)=(\widetilde M,\widetilde g)$ and suppose that
$\hiso(X)=m-1$.  Compactness of $M$ gives a constant $A\geq1$ such that
\[
  -A^2\leq\secg_{g_X}\leq-1.
\]
Fix $\xi_0\in\partial_\infty X$ and choose an isoperimetric minimizing
sequence $\{\Omega_j\}$.  Corollary~\ref{cor:limiting-measure} gives a
weak limit $\mu$ of the associated probability measures on $Z$ such that
\[
  \supp\mu\subseteq D^{-1}(0).
\]
Proposition~\ref{prop:full-support} gives $\supp\mu=Z$.  Therefore
$D^{-1}(0)=Z$, and the definition of $D$ yields
\begin{equation}\label{eq:busemann-laplacian-equality}
  \Delta B_\xi(x)=m-1
  \qquad
  \text{for every }(x,\xi)\in X\times\partial_\infty X.
\end{equation}

Fix $\xi\in\partial_\infty X$.  Proposition~
\ref{prop:busemann-comparison} gives $B_\xi\in C^2(X)$ and
$|\nabla B_\xi|=1$.  Equation~\eqref{eq:busemann-laplacian-equality} and
Lemma~\ref{lem:busemann-rigidity}, applied to $B_\xi$, give
$(X,g_X)\cong\Hh^m(-1)$.
\end{proof}

\section{Relation with the Cheeger inequality}
\label{sec:p-comparison}

Write
\[
  h(X):=\hiso(X).
\]
For \(1\leq p<\infty\), set
\[
  \lambda_{1,p}(X)
  :=\inf_{0\neq u\in C_c^\infty(X)}
  \frac{\displaystyle\int_X|\nabla u|^p\,dV}
       {\displaystyle\int_X|u|^p\,dV}.
\]
For $p>1$, this is the variational bottom associated with $-\Delta_p$,
where
\[
  \Delta_pu:=\diver\bigl(|\nabla u|^{p-2}\nabla u\bigr).
\]
\begin{proposition}
\label{prop:standard-cheeger-inequality}
Let \(X\) be a complete Riemannian manifold.  Then
\begin{equation}\label{eq:lambda-one-cheeger}
  \lambda_{1,1}(X)=h(X),
\end{equation}
and, for every \(p\in(1,\infty)\),
\begin{equation}\label{eq:cheeger-inequality-lambda-p}
  \lambda_{1,p}(X)\geq
  \left(\frac{h(X)}{p}\right)^p.
\end{equation}
\end{proposition}

\begin{proof}
For \(u\in C_c^\infty(X)\), the coarea formula and the definition of
\(h(X)\) give
\[
\begin{aligned}
  \int_X|\nabla u|\,dV
  &=\int_0^\infty
    \area\bigl(\partial\{|u|>t\}\bigr)\,dt\\
  &\geq h(X)\int_0^\infty
    \vol\bigl(\{|u|>t\}\bigr)\,dt
   =h(X)\int_X|u|\,dV.
\end{aligned}
\]
Hence \(\lambda_{1,1}(X)\geq h(X)\).  Conversely, if
\(\Omega\Subset X\) has smooth boundary, standard smooth
approximations of \(\mathbf 1_\Omega\) satisfy
\[
  \int_Xu_\varepsilon\,dV\longrightarrow\vol(\Omega),
  \qquad
  \int_X|\nabla u_\varepsilon|\,dV
  \longrightarrow\area(\partial\Omega).
\]
Taking the infimum first over test functions and then over \(\Omega\)
proves \eqref{eq:lambda-one-cheeger}.

For \(p>1\), apply the preceding coarea estimate to \(|u|^p\) and use
H\"older's inequality:
\[
\begin{aligned}
  h(X)\int_X|u|^p\,dV
  &\leq p\int_X|u|^{p-1}|\nabla u|\,dV\\
  &\leq p
  \left(\int_X|u|^p\,dV\right)^{(p-1)/p}
  \left(\int_X|\nabla u|^p\,dV\right)^{1/p}.
\end{aligned}
\]
After division and taking the infimum, this is
\eqref{eq:cheeger-inequality-lambda-p}.
\end{proof}

Combining Proposition~\ref{prop:standard-cheeger-inequality} with Yau's
bound $h(X)\geq m-1$ gives:

\begin{corollary}
\label{cor:p-laplacian-equality}
Under the hypotheses of Theorem~\ref{thm:main}, for every
\(p\in(1,\infty)\),
\begin{equation}\label{eq:p-mckean-bound}
  \lambda_{1,p}(X)\geq\left(\frac{m-1}{p}\right)^p.
\end{equation}
Equality holds in \eqref{eq:p-mckean-bound} if and only if
\((X,g_X)\cong\Hh^m(-1)\).
\end{corollary}

\begin{proof}
Combining \eqref{eq:yau-lower} with
\eqref{eq:cheeger-inequality-lambda-p} proves \eqref{eq:p-mckean-bound};
see also \cite{Poliquin2014}*{Proposition~4.5} and
\cite{CarvalhoCavalcante2022}*{Corollary~1.1}.  If equality holds,
then
\[
  \left(\frac{m-1}{p}\right)^p
  =\lambda_{1,p}(X)
  \geq\left(\frac{h(X)}{p}\right)^p
  \geq\left(\frac{m-1}{p}\right)^p.
\]
Thus \(h(X)=m-1\), and Theorem~\ref{thm:main} gives
\((X,g_X)\cong\Hh^m(-1)\).

Conversely, let \(X=\Hh^m(-1)\) and put \(\alpha=(m-1)/p\).  The
lower bound has already been proved.  In geodesic polar coordinates,
choose smooth radial functions \(u_R\) which agree with
\(e^{-\alpha r}\) on \(1\leq r\leq R\), are fixed near the origin,
and vanish for \(r\geq R+1\).  They may be chosen so that, on the outer
transition region,
\[
  |u_R(r)|+|u_R'(r)|\leq C e^{-\alpha R},
\]
where $C$ is independent of $R$.  Since
\[
  e^{-p\alpha r}\sinh^{m-1}r
  =e^{-(m-1)r}\sinh^{m-1}r
  \longrightarrow2^{-(m-1)},
\]
the contributions of the transition regions remain bounded, whereas
both principal integrals grow linearly in \(R\).  Consequently,
\[
  \frac{\int_X|\nabla u_R|^p\,dV}
       {\int_X|u_R|^p\,dV}
  \longrightarrow\alpha^p.
\]
Hence equality holds on hyperbolic space.
\end{proof}

\appendix
\section{Tangential approximation along horospheres}
\label{sec:horosphere-tangent-approximation}

The Gauss--Green argument in Lemma~\ref{lem:horosphere-div} requires a
tangency-preserving approximation.

\begin{lemma}
\label{lem:horosphere-tangent-approximation}
Let $(X^m,g_X)$ be a smooth Riemannian manifold, and let $B\in C^2(X)$
satisfy $\lvert\nabla B\rvert=1$.  Set
$\mathcal H:=\{H_t:t\in\R\}$, where $H_t=B^{-1}(t)$, and let $Y$ be a
continuous vector field on $X$ such that
$\langle Y,\nabla B\rangle=0$.  Assume that $Y|_{H_t}$ is $C^1$ for
every $t\in\R$.  For $x\in X$ and $V\in T_xH_{B(x)}$, set
\[
  (\nabla^{\mathcal H}Y)_x(V)
  :=\nabla^{H_{B(x)}}_V(Y|_{H_{B(x)}}),
\]
where $\nabla^{H_{B(x)}}$ is the Levi--Civita connection of the induced
metric on $H_{B(x)}$.  Assume that $\nabla^{\mathcal H}Y$ is continuous
on $X$, and set
\[
  \diver_{\mathcal H}Y(x)
  :=\operatorname{tr}_{T_xH_{B(x)}}(\nabla^{\mathcal H}Y)_x.
\]
This function is continuous.  For every smooth domain
$\Omega\Subset X$, there exists a sequence of vector fields
\[
  Y_k\in C_c^1(X,TX)
\]
such that
\[
  dB(Y_k)=0
\]
and
\[
  Y_k\longrightarrow Y,
  \qquad
  \diver_XY_k\longrightarrow\diver_{\mathcal H}Y
\]
uniformly on $\overline\Omega$.  If $B$ is smooth, the fields $Y_k$ may
be chosen in $C_c^\infty(X,TX)$.
\end{lemma}

\begin{proof}
Fix $\Omega\Subset X$.  Since $|\nabla B|=1$, the level sets of $B$ are
regular hypersurfaces.  The $C^2$ implicit function theorem gives adapted
coordinates
\[
  (t_i,y_i^1,\ldots,y_i^{m-1})
\]
on a neighborhood $U_i$, where $t_i=B$.  These charts are of class
$C^2$, so their coordinate vector fields are $C^1$.
Choose finitely many such charts $U_1,\ldots,U_\ell$ covering a
neighborhood of $\overline\Omega$, together with open sets
\[
  V_i\Subset U_i,
  \qquad
  \overline\Omega\subset\bigcup_{i=1}^\ell V_i.
\]

Tangency to the level sets means that on $U_i$ the field $Y$ has the form
\[
  Y=\sum_{a=1}^{m-1}Y_i^a(t_i,y_i)
  \frac{\partial}{\partial y_i^a}.
\]
If $\Gamma_{i,bc}^a$ are the Christoffel symbols of the induced metrics
on the level sets, then
\[
  \bigl(\nabla^{H_{t_i}}_{\partial_{y_i^b}}Y\bigr)^a
  =\partial_{y_i^b}Y_i^a
   +\sum_{c=1}^{m-1}\Gamma_{i,bc}^aY_i^c.
\]
The functions $\Gamma_{i,bc}^a$ are continuous.  Hence the continuity of
$\nabla^{\mathcal H}Y$ implies that every
$\partial_{y_i^b}Y_i^a$ is continuous.
Choose $\chi_i\in C_c^1(U_i)$ with $\chi_i=1$ on $\overline V_i$, and
set $G_i^a=\chi_iY_i^a$.  In the adapted coordinates, extend $G_i^a$ by
zero to $\mathbb R^m$.  The extension is continuous, and its derivatives
in the leaf variables are continuous and compactly supported because
\[
  \partial_{y_i^b}G_i^a
  =(\partial_{y_i^b}\chi_i)Y_i^a
  +\chi_i\partial_{y_i^b}Y_i^a.
\]

Let $\rho$ be a standard mollifier on $\mathbb R^m$, and set
\[
  \rho_\varepsilon(z)=\varepsilon^{-m}\rho(z/\varepsilon),
  \qquad
  G_{i,\varepsilon}^a=\rho_\varepsilon*G_i^a.
\]
No derivative of $Y$ in the $t_i$-direction is needed: convolution of the
continuous compactly supported coefficients is defined in all variables,
while differentiation is taken only in the leaf variables.  In particular,
\[
  \partial_{y_i^b}G_{i,\varepsilon}^a
  =\rho_\varepsilon*\partial_{y_i^b}G_i^a.
\]
It follows that, uniformly on $\overline V_i$,
\begin{equation}\label{eq:appendix-local-convergence}
  G_{i,\varepsilon}^a\longrightarrow Y_i^a,
  \qquad
  \partial_{y_i^b}G_{i,\varepsilon}^a
  \longrightarrow\partial_{y_i^b}Y_i^a.
\end{equation}

Define on $U_i$
\[
  Z_{i,\varepsilon}
  =\sum_{a=1}^{m-1}G_{i,\varepsilon}^a
  \frac{\partial}{\partial y_i^a}.
\]
This is a $C^1$ vector field.  It has no
$\partial_{t_i}$-component, and therefore
\[
  dB(Z_{i,\varepsilon})=0.
\]
Thus the local regularization preserves tangency to every level set of $B$.

Write the Riemannian volume form in these coordinates as
\[
  dV=Q_i(t_i,y_i)\,dt_i\,dy_i^1\cdots dy_i^{m-1}.
\]
The coarea formula and $|\nabla B|=1$ show that the induced volume form
on $H_{t_i}$ is $Q_i(t_i,y_i)\,dy_i$.  Hence
\[
  \diver_XZ_{i,\varepsilon}
  =\frac1{Q_i}\sum_{a=1}^{m-1}
  \partial_{y_i^a}(Q_iG_{i,\varepsilon}^a),
\]
whereas
\[
  \diver_{\mathcal H}Y
  =\frac1{Q_i}\sum_{a=1}^{m-1}
  \partial_{y_i^a}(Q_iY_i^a).
\]
The functions $Q_i$ and their leafwise first derivatives are continuous.
Hence \eqref{eq:appendix-local-convergence} gives, uniformly on
$\overline V_i$,
\begin{equation}\label{eq:appendix-local-divergence-convergence}
  Z_{i,\varepsilon}\longrightarrow Y,
  \qquad
  \diver_XZ_{i,\varepsilon}
  \longrightarrow\diver_{\mathcal H}Y.
\end{equation}

Choose an open neighborhood $W$ of $\overline\Omega$ with
$\overline W\subset\bigcup_{i=1}^\ell V_i$.  Choose $C^1$ functions
$\phi_i$ such that
\[
  \phi_i\in C_c^1(V_i),
  \qquad
  \sum_{i=1}^\ell\phi_i=1
  \quad\text{on }\overline W.
\]
Choose $\varepsilon_k\downarrow0$ so that the preceding convergences hold
simultaneously in all the finitely many charts, and set
\[
  Z_{i,k}=Z_{i,\varepsilon_k},
  \qquad
  Y_k=\sum_{i=1}^\ell\phi_iZ_{i,k}.
\]
Each product $\phi_iZ_{i,k}$ extends by zero outside $U_i$.  Thus $Y_k$
is a globally defined compactly supported $C^1$ vector field, and
\[
  dB(Y_k)=\sum_{i=1}^\ell\phi_i\,dB(Z_{i,k})=0.
\]
Since $Y=\sum_i\phi_iY$ on $W$,
\[
  Y_k-Y=\sum_{i=1}^\ell\phi_i(Z_{i,k}-Y)
  \longrightarrow0
\]
uniformly on $\overline\Omega$.

The product rule and $\sum_i\nabla\phi_i=0$ on $W$ give
\begin{align*}
  \diver_XY_k-\diver_{\mathcal H}Y
  &=\sum_{i=1}^\ell\phi_i
  \bigl(\diver_XZ_{i,k}-\diver_{\mathcal H}Y\bigr)\\
  &\quad+\sum_{i=1}^\ell
  \langle\nabla\phi_i,Z_{i,k}-Y\rangle.
\end{align*}
Every term on the right tends uniformly to zero on $\overline\Omega$ by
\eqref{eq:appendix-local-divergence-convergence}.  Therefore
\[
  \diver_XY_k\longrightarrow\diver_{\mathcal H}Y
\]
uniformly on $\overline\Omega$.  If $B$ is smooth, all adapted charts,
cutoffs, and partitions may be chosen smooth, and the same construction
gives $Y_k\in C_c^\infty(X,TX)$.
\end{proof}

\begin{acknowledgements}
The defect quantity and elliptic-operator argument in the earlier paper of
Kuntao Jin and Bo Zhu motivated this work (see
\cite{JinZhu2026McKean}).  The earlier paper uses weak stable sets.  Here the
Busemann parameter is fixed and the argument is carried out on strong stable
leaves, using strong stable divergence, plaque disintegration, and support
propagation. The authors conceived the reformulation of (2.10) in which the integral of the leafwise divergence along the strong stable leaves vanishes, and used a generative AI tool to develop this idea into a proof that \(\operatorname{supp}\mu=Z\). The AI-generated proof was subsequently discarded and replaced by a shorter proof written entirely by the authors. No generative AI tool was used in any other part of the research or in the preparation of the manuscript.
\end{acknowledgements}

\end{document}